\documentclass[11pt,a4paper]{article}
\usepackage[utf8]{inputenc}
\usepackage{amsmath}
\usepackage{amsfonts}
\usepackage{tikz}
\usepackage{amssymb}
\usepackage{amsthm}
\usepackage{graphicx}
\usepackage{subfig}
\usepackage{hyperref}
\usepackage{color}
\usepackage{comment}
\usepackage{appendix}
\newtheorem{assume}{Assumption}
\newtheorem{theorem}{Theorem}
\newtheorem{lemma}{Lemma}
\newtheorem{remark}{Remark}
\newtheorem{corollary}{Corollary}
\newtheorem{definition}{Definition}
\title{The PML method for calculating the propagative wave numbers of electromagnetic wave in periodic structures}

\author{Lide Cai\thanks{\footnotesize Department of Mathematical Sciences, Tsinghua University, Beijing 100084, China. 
(cld19@mails.tsinghua.edu.cn).}
\and Junqing Chen\thanks{\footnotesize
Department of Mathematical Sciences, Tsinghua University, Beijing
100084, China. The work of this author was partially supported by National Key R\&D Program of China 2019YFA0709600, 2019YFA0709602. (jqchen@tsinghua.edu.cn).}
\and Yanpeng Gao\thanks{\footnotesize
Department of Mathematical Sciences, Tsinghua University, Beijing
100084, China. (gaoyp24@mails.tsinghua.edu.cn)
}
}

\begin{document}
\maketitle
\begin{abstract}
	When the electromagnetic wave is incident on the periodic structures, in addition to the scattering field, some guided modes that are traveling in the periodic medium could be generated. In the present paper, we study the calculation of guided modes. We formulate the problem as a nonlinear eigenvalue problem in an unbounded periodic domain. Then we use perfectly matched layers to truncate the unbounded domain, recast the problem to a quadratic eigenvalue problem, and prove the approximation property of the truncation. Finally, we formulate the quadratic eigenvalue problem to a general eigenvalue problem, use the finite element method to discrete the truncation problem, and show numerical examples to verify theoretical results. 
\end{abstract}
%{\footnotesize{\bf Mathematics Subject Classification}(MSC2020): 35B27, 65N25, 78A40,78A50, 78M10}
{{\bf Mathematics Subject Classification}(MSC2020): 35P30, 35B27, 78M10, 47A56}\\
{{\bf Keywords:} electromagnetic wave in periodic structure, propagating mode, open waveguide, perfectly matched layer, quadratic eigenvalue}

\section{Introduction}

 The scattering problem in periodic structures has long been an important topic in the numerical analysis community.  When an electromagnetic wave is incident on the periodic structures, in addition to the scattering part, there are certain guided modes which travel along the periodic structure. 
In the present paper, we consider the periodic scattering problem in $\mathbb{R}^2$. We assume that the wave number $k$ is  homogeneous in $\mathbb{R}^2$ and the refractive index $\gamma(x) \in L^{\infty} ( \mathbb{R}^2)$ is $\Lambda$-periodic in the $x_1$ direction. That is to say,
$$\gamma(x_1+\Lambda,x_2)=\gamma(x_1,x_2),$$
and $\gamma(x_1,x_2)=1$ for $|x_2|>h_0$ for some $h_0>0$. We let $\Omega=\{(x_{1},x_{2})|-\frac{\Lambda}{2}<x_{1}<\frac{\Lambda}{2},x_{2}\in\mathbb{R}\}$ be the Wigner-Seitz cell of the periodic structure \cite{Lechleiter}. Then we can state the scattering problem as follows: Given an incident wave $u^{inc}$, the governing equation of the total field $u$ is 
\begin{equation}
	\Delta u + k^2 \gamma(x)u = 0, \label{Hequ}
\end{equation} 
and $u=u^{inc}+u^s+u_{guide}$, where $u^{inc}$ is the incident wave, $u^s$ is the scattering wave, and $u_{guide}$ is the possible guided mode which will be defined and explained later. See Figure \ref{pbdry} for a demonstration of the problem.
 \begin{figure}[ht]
	\centering{
		{\includegraphics[width=0.8\textwidth]{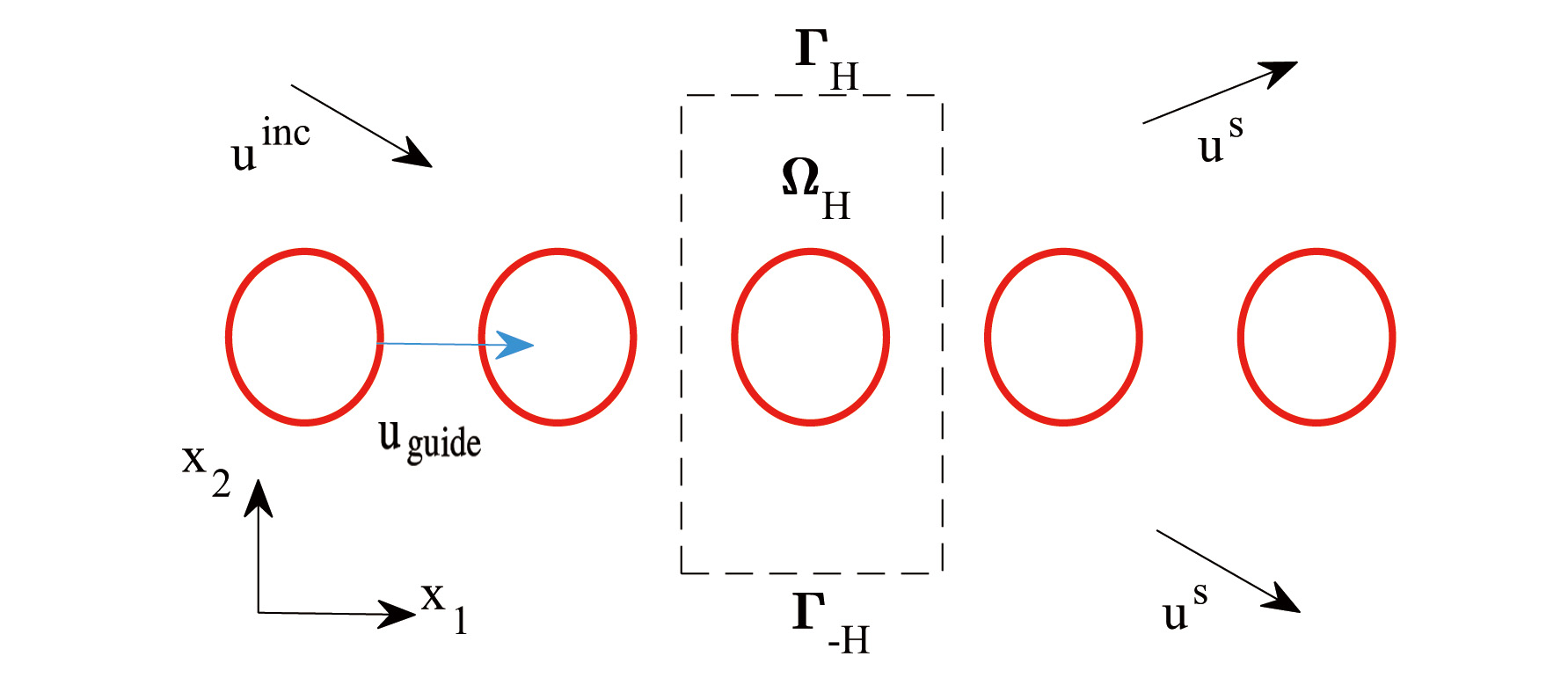}}
	}
	\caption{The description of the periodic structures and the guided mode.}\label{pbdry}
\end{figure}
 When $u^{inc}$ is $\alpha$-quasi-periodic, namely,
\begin{equation*}
	u^{inc}(x_{1}+\Lambda,x_{2})=e^{i\alpha\Lambda}u^{inc}(x_{1},x_{2}), 
\end{equation*}
one can expect the above equation to be uniquely solvable for some $u \in H^{1}_{\alpha, loc}(\mathbb{R}^{2})$, where this space consists of functions in $H^{1}_{loc}(\mathbb{R}^{2})$ that are $\alpha$-quasi-periodic in $x_1$. The solution $u$ is required to satisfy a Rayleigh radiation condition \cite{Bao2}. Specifically, for some $H > h_0>0$, it admits the following expansions: in the upper region $\Omega_{H}^{+} = \{ x=(x_1,x_2)\in\Omega \mid x_2 > H \}$,
    \begin{equation}
    u - u^{inc} = \sum_{n \in \mathbb{Z}} u^{+}_{n} e^{i\alpha_n x_1 + i\beta_n (x_2 - H)}, \label{uRay0}
    \end{equation}
    and in the lower region $\Omega_{H}^{-} = \{ x=(x_1,x_2)\in\Omega \mid x_2 < -H \}$,
    \begin{equation}
    u  = \sum_{n \in \mathbb{Z}} u^{-}_{n} e^{i\alpha_n x_1 - i\beta_n (x_2 + H)}. \label{lRay0}
    \end{equation}
Here
\begin{eqnarray*}
	\alpha_{n}=\alpha+\frac{2\pi}{\Lambda} n, n\in\mathbb{Z},~
	\beta_{n}=\sqrt{k^2-\alpha_{n}^{2}}.
\end{eqnarray*}
For the square root function, we take the branch cut along the negative imaginary axis. Throughout this paper, we denote by $B_{\alpha} := \{n \in \mathbb{Z} \mid |\alpha_{n}| \leq k\}$ and $U_{\alpha} := \{n \in \mathbb{Z} \mid |\alpha_{n}| > k\}$.
The uniqueness of the solution implies that $u_{guide}$ vanishes. However, as was shown in \cite{Kirsch1}, for certain $\gamma(x)$ and $k$ that satisfy our assumption, there are possibly $\alpha$s such that the corresponding Rayleigh radiating condition is not enough to guarantee the uniqueness of a solution to \eqref{Hequ}. These $\alpha$s are called propagative wave numbers, and the corresponding $u_{guide}$s are called guided modes which destroy the uniqueness; see the detailed definition in Definition \ref{def1}.

The mathematical analysis on the above scattering problem has attracted the attention of numerous researchers, whose main concern is to provide a suitable description of a radiation condition that guarantees the uniqueness of solution to the scattering problem while including the guided modes.
Starting from \cite{Fliss2}, using the limiting absorption principle with the Floquet-Bloch transformation, the characterization of guided modes in a closed waveguide problem was performed, while \cite{Kirsch1} was the first to apply the transformation to the case of a semi-open waveguide problem to derive the suitable scattering condition, and further development in the case of scattering by locally perturbed periodic structures can be found in \cite{Kirsch2,Kirsch3,Kirsch4,Fliss2,Fliss3,Lechleiter,Nosich}.

In the context of numerical algorithms, there has been a growing interest in developing reliable numerical methods for scattering by periodic and locally perturbed periodic structures. Examples include schemes based on the Floquet–Bloch transformation \cite{Lechleiter}, the volume integral equation method \cite{Furuya}, and the boundary integral equation method combined with a specially designed perfectly matched layer (PML) and a Neumann-to-Dirichlet (NtD) operator \cite{Yu}. In addition, recent work has focused on the design of PML methods for scattering by locally perturbed periodic structures \cite{Kirsch5}, where a carefully chosen contour is employed to avoid propagative wave numbers in the inverse Floquet transformation. However, since the PML approach analyzed in \cite{Kirsch5} circumvents the computation of propagative wave numbers, the accurate calculation of these wave numbers and their corresponding guided modes remains an important and open problem.

On the other hand, for the closed periodic waveguide problem, where a Dirichlet/Neumann boundary condition is imposed on the total wave $u$ in the $x_{2}$-direction, the computation of corresponding eigenmodes for $\alpha$s has been thoroughly studied. These eigenmodes are referred to as guided modes and have attracted significant interest both mathematically \cite{Kuchment1,Kuchment2} and numerically \cite{Engstrom1,Fliss1,Zhang3}. Furthermore, in the case of the locally perturbed close periodic waveguide, the numerical computation of guided waves was identified as a necessary step in \cite{Zhang2}. 

The main contribution of our work is therefore a numerical method dedicated to computing such propagative wave numbers in the open periodic waveguide problem. This is achieved by designing a suitable PML truncation that transforms the original open problem into a closed periodic waveguide problem, to which well-established numerical methods can be applied. Furthermore, a novel analysis of the approximation error introduced by the truncation is thoroughly presented.

To be more specific, the approximation analysis of the propagative wave numbers consists of two main ingredients. 
The first is the approximation analysis of the PML truncation in operator level. We prove the exponential convergence property of the PML truncation to the original propagative wavenumber problem by a novel analysis of the PML DtN operator. The second ingredient is the Gohberg–Sigal theory for operator-valued functions \cite{Ammari, Ammari2010}, which allows us to establish the perturbation property of the propagative values under PML truncation. 

The structure of the paper is as follows. In Section \ref{sect2}, we recall the definition of the propagative wave numbers and truncate the problem to a bounded domain, where the nonlinearity arising from the Rayleigh expansion is encapsulated by the Dirichlet-to-Neumann (DtN) operator. We further analyze the basic properties of the propagative wave numbers and the analyticity of the corresponding operator in this section. In Section \ref{sect3}, we circumvent the nonlinearity of the Rayleigh expansion by introducing a PML truncation and define a PML DtN operator, which allows us to reformulate the problem as a second-order eigenvalue problem. We also investigate the perturbation induced by the PML truncation at the operator level and establish an exponential decay property in the operator norm with respect to the PML parameter. In Section \ref{sect4}, we provide a detailed analysis of the multiplicity of the propagative wave numbers. Equipped with these tools, in Section \ref{sect5} we use the exponential decay of the operator perturbation to derive the exponential decay of the perturbation of the propagative wave numbers. A novel trace inequality is proved here, and the inspiration for this analysis is drawn from \cite{Ammari}, which deals with the spectral analysis of perturbations caused by small inclusions. In Section \ref{sect6}, we present a finite element approximation analysis of the truncated problem, completing the analysis of the discrete scheme. Finally, in Section \ref{sect7}, we report several numerical experiments to demonstrate the reliability and effectiveness of our proposed method.
%%%%%%%%%%%%%%%%%%%%%%%%%%%%%%%%%%%%%%%%%%%%%%%%%%%%%%%
\section{Problem formulation on a bounded domain}\label{sect2}

We first recall the definition of propagative wave numbers \cite{Kirsch1} (or the exceptional values in \cite{Kirsch2}).
\begin{definition}\label{def1}
	Given a wavenumber $k>0$ and $\Lambda$-periodic refractive index $\gamma(x)$, we call $\alpha\in[-\frac{\pi}{\Lambda},\frac{\pi}{\Lambda}]$ a \textbf{propagative wave number} if the following Helmholtz equation has a nonzero $\alpha$-quasi-periodic solution $u\in H^{1}_{\alpha,loc}(\mathbb{R}^{2})$,
	\begin{equation*}
		\Delta u + k^2 \gamma(x)u = 0. %\label{Hequ}
	\end{equation*} 
	which satisfies a Rayleigh propagation condition, namely, for some $H>h_0>0$, in $\Omega_{H}^{+}=\{x=(x_{1},x_{2})\in\Omega|x_{2}>H\}$
	\begin{equation}
		u(x_{1},x_{2}) = \sum_{n\in \mathbb{Z}}u^{+}_{n}e^{i\alpha_{n}x_{1}+i\beta_{n}(x_{2}-H)}, \label{uRay}
	\end{equation}
	and in $\Omega_{H}^{-}=\{x=(x_{1},x_{2})\in\Omega|x_{2}<-H\}$
	\begin{equation}
		u(x_{1},x_{2}) = \sum_{n\in \mathbb{Z}}u^{-}_{n}e^{i\alpha_{n}x_{1}-i\beta_{n}(x_{2}+H)}. \label{lRay}
	\end{equation}
	We call such $u$ the \textbf{guided mode} corresponding to $\alpha$.
\end{definition}
\subsection{Variational problem with DtN operator}
The set of all propagative wave numbers is denoted by $P(k)$. Let $\Gamma_{\pm H}=\lbrace(x_{1},x_{2}):-\frac{\Lambda}{2}<x_{1}<\frac{\Lambda}{2}, x_{2}=\pm H\rbrace$.  
To handle the problem in a bounded region, we first note that for any $u\in H^1_{\alpha,loc}(\mathbb{R}^2)$, the traces $u(x_{1},\pm H)$ are well-defined and belong to $H^{1/2}(\Gamma_{\pm H})$. Moreover, these traces are $\alpha$-quasi-periodic in the $x_{1}$ direction. We then define the DtN operators $T^\pm_\alpha: H^{1/2}(\Gamma_{\pm H})\rightarrow H^{-1/2}(\Gamma_{\pm H})$ by
\begin{eqnarray*}
	T^+_\alpha u=\sum_{n\in\mathbb{Z}}i\beta_{n}u_{n}^{+}e^{i\alpha_{n}x_{1}}, \, x\mbox{ on }\Gamma_H 
\end{eqnarray*}
and 
\begin{eqnarray*}
	T^-_\alpha u=\sum_{n\in\mathbb{Z}}i\beta_{n}u_{n}^{-}e^{i\alpha_{n}x_{1}},\, x\mbox{ on } \Gamma_{-H},
\end{eqnarray*}
with $u^{\pm}_{n}=\frac{1}{\Lambda}\int_{-\frac{\Lambda}{2}}^{\frac{\Lambda}{2}}u(x_1,\pm  H)e^{-i\alpha_n x_1}dx_1,n\in\mathbb{Z}$ the Fourier coefficients corresponding to $u(x_{1},\pm H)$. We denote $\Omega_H=\{(x_1,x_2)|-\frac{\Lambda}{2}<x_1<\frac{\Lambda}{2}, -H<x_2<H\}$. Let $H^1_\alpha(\Omega_H)=\{u|_{\Omega_H}| u\in H^1_{\alpha,loc}(\mathbb{R}^2)\}$ and $H^1_{per}(\Omega_H)=\{u\in H^1(\Omega_H)| u(-\Lambda/2,x_2)=u(\Lambda/2,x_2)\}$.
Then with the help of operators $T^{\pm}_\alpha$, 
we can formulate \eqref{Hequ}-\eqref{lRay} in $\Omega_{H}$ to the following weak problem: find $\alpha \in[-\frac{\pi}{\Lambda},\frac{\pi}{\Lambda}]$ and $0\neq u\in H^1_\alpha(\Omega_H)$, such that for every $v \in H^1_\alpha(\Omega_H)$, 
\begin{equation}
	\int_{\Omega_{H}}\big(\nabla u \cdot\nabla \bar{v} - k^{2}\gamma(x)u\bar{v}\big)dx -\int_{\Gamma_{H}}T_{\alpha}^{+}u\bar{v}ds-\int_{\Gamma_{-H}}T_{\alpha}^{-}u\bar{v}ds=0.\label{wk1}
\end{equation}
Since $u,v\in H^1_\alpha(\Omega_H)$, we can introduce $\phi,\psi \in H^1_{per}(\Omega_H)$ such that $u=\phi e^{i\alpha x_{1}},v=\psi e^{i\alpha x_{1}}$, let 
\begin{eqnarray*}
	&&a_{\alpha}(\phi,\psi)=\int_{\Omega_{H}}\big(\nabla \phi \cdot\nabla \bar{\psi} -2i\alpha\frac{\partial \phi}{\partial x_{1}}\bar\psi-(k^{2}\gamma(x)-\alpha^{2})\phi\bar{\psi}\big)dx\\
    &&\qquad\qquad\qquad-\sum_{n\in\mathbb{Z}}i\Lambda\beta_{n}(\phi^{+}_{n}\bar{\psi}^{+}_{n}+\phi^{-}_{n}\bar{\psi}^{-}_{n}). 
\end{eqnarray*}
Here, $\phi_{n}^{\pm}=\frac{1}{\Lambda}\int_{-\frac{\Lambda}{2}}^{\frac{\Lambda}{2}}\phi(x_1,\pm  H)e^{-i\frac{2\pi}{\Lambda}n x_1}dx_1,\psi_{n}^\pm=\frac{1}{\Lambda}\int_{-\frac{\Lambda}{2}}^{\frac{\Lambda}{2}}\psi(x_1,\pm  H)e^{-i\frac{2\pi}{\Lambda}n x_1}dx_1,n\in\mathbb{Z}$ are the Fourier coefficients corresponding to the trace of $\phi,\psi$ on $\Gamma_{\pm H}$. Thus, our problem of computing propagative wave numbers can be recast as follows: find $\alpha\in[-\frac{\pi}{\Lambda},\frac{\pi}{\Lambda}]$  and $0\neq \phi\in H^1_{per}(\Omega_H)$ such that
\begin{equation}
	a_\alpha(\phi,\psi)=0, \forall \psi\in H^1_{per}(\Omega_H).\label{oripblm}
\end{equation}
Further, we define an inner product on $H^{1}_{per}(\Omega_{H})$ by
\begin{equation}
	(\phi,\psi)_{*}=\int_{\Omega_{H}}\nabla \phi\cdot \nabla \bar{\psi} dx + 2\pi\sum_{n\in\mathbb{Z}}\sqrt{n^{2}+1}(\phi_{n}^{+}\bar{\psi}^{+}_{n}+\phi^{-}_{n}\bar{\psi}^{-}_{n}).\label{starnorm}
\end{equation}
and denote by $\|\cdot\|_*$ the norm induced by the inner product. Analogous to Lemma 2.3 in \cite{ChenWu}, we can show that there exists a constant $C_{1,*}>0$ such that  
$$
\Vert \phi \Vert_* \le C_{1,*} \Vert \phi \Vert_{H^1(\Omega_H)}.
$$  
Define the Fourier coefficient  
$$
\phi_n(x_2)=\frac{1}{\Lambda}\int_{-\frac{\Lambda}{2}}^{\frac{\Lambda}{2}}\phi(x_1,x_2)e^{-i\frac{2\pi}{\Lambda}n x_1}dx_1,
$$  
then we have  
\begin{align*}
   |\phi_n^{\pm}|^2 &=|\phi_n(x_2)|^2+\int_{x_2}^{\pm H}\frac{d}{dx_2}|\phi_n(s)|^2 ds \\
   &\ge  |\phi_n(x_2)|^2-\int_{-H}^H 2|\phi_n(x_2)|\left|\frac{d}{dx_2}\phi_n(x_2)\right|dx_2. 
\end{align*}
Integrating both sides with respect to $x_2$ over $[-H,H]$ yields  
\begin{align*}
    |\phi_n^{\pm}|^2
   &\ge \int_{-H}^H\bigg( \frac{1}{2H}|\phi_n(x_2)|^2- 2|\phi_n(x_2)|\left|\frac{d}{dx_2}\phi_n(x_2)\right|\bigg)dx_2 \\ 
   & \ge  \int_{-H}^H\bigg( \frac{1}{4H}|\phi_n(x_2)|^2- 4H\left|\frac{d}{dx_2}\phi_n(x_2)\right|^2\bigg)dx_2.
\end{align*}
Consequently, there exist positive constants $C_{2,*}^{0}$ and $C_{2,*}^{1}>1$ such that  
\begin{align*}
  & \quad \, \, 2\pi \sum_{n \in \mathbb{Z}}\sqrt{1+n^2}\big(|\phi_n^+|^2+|\phi_n^-|^2\big) \\
   &\ge  4\pi\sum_{n \in \mathbb{Z}}\int_{-H}^H\bigg( \frac{1}{4H}|\phi_n(x_2)|^2- 4H\left|\frac{d}{dx_2}\phi_n(x_2)\right|^2\bigg)dx_2   \\
   & \ge  C_{2,*}^{0} \Vert \phi \Vert_{L^2(\Omega_H)} -C_{2,*}^{1} \Vert \nabla \phi \Vert_{L^2(\Omega_H)}.
\end{align*}
Therefore,  
$$
\Vert \phi \Vert_*^2 \ge \frac{1}{2} \Vert \nabla \phi \Vert_{L^2(\Omega_H)}^2+\frac{C_{2,*}^{0}}{2C_{2,*}^{1}} \Vert \phi \Vert_{L^2(\Omega_H)}.
$$  
Thus the norm $\Vert \cdot \Vert_*$ is equivalent to the standard Sobolev norm $\Vert \cdot \Vert_{H^1(\Omega_H)}$. For a normed linear space $ \mathcal{H} $, we denote by $ \mathcal{L}(\mathcal{H}) $ the space of all bounded linear operators from $ \mathcal{H} $ to itself. We consider the Banach space of all bounded linear operators $\mathcal{L}(H^1_{per}(\Omega_{H}))$, equipped with the operator norm 
induced by $\Vert \cdot\Vert _{*}$. Now we have
\begin{theorem}\label{Thm:1}
	For $k > 0$ and $\alpha \notin 
\bigl\{z \in \mathbb{C} \mid 
\Re(k^{2} - (z + \frac{2\pi n}{\Lambda})^{2}) = 0,\,
\Im(k^{2} - (z + \frac{2\pi n}{\Lambda})^{2}) < 0 \text{ for some } n \in \mathbb{Z}\bigr\}$, there is  a compact operator $K(\alpha)$ in $\mathcal{L}(H^1_{per}(\Omega_{H}))$, such that for any $\phi,\psi\in H^{1}_{per}(\Omega_{H})$,
	\begin{equation*}
		a_{\alpha}(\phi,\psi)=(D(\alpha)\phi,\psi)_{*}=((I-K(\alpha))\phi,\psi)_{*}.
	\end{equation*}
Here $I$ denotes the identity operator. 
\end{theorem}
\begin{proof}
	From the definition of $a_{\alpha}(\phi,\psi)$, we readily to obtain that 
	\begin{eqnarray*}
		a_{\alpha}(\phi,\psi)&=&\int_{\Omega_{H}}\nabla \phi \cdot\nabla \bar{\psi} dx+2\pi\sum_{n\in\mathbb{Z}}\sqrt{n^2+1} (\phi_{n}^{+}\bar\psi_{n}^{+}+\phi_{n}^{-}\bar{\psi}_{n}^{-})\\
		&-&\int_{\Omega_{H}}\big(2i\alpha\frac{\partial \phi}{\partial x_{1}}\bar\psi+(k^{2}\gamma(x)-\alpha^{2})\phi\bar{\psi}\big)dx\\
		&-&\sum_{n\in\mathbb{Z}}(i\Lambda\beta_{n}+2\pi\sqrt{n^2+1} )\phi^{+}_{n}\bar{\psi}^{+}_{n}-\sum_{n\in\mathbb{Z}}(i\Lambda\beta_{n}+2\pi\sqrt{n^2+1} )\phi^{-}_{n}\bar{\psi}^{-}_{n}.
	\end{eqnarray*}
	Let
	\begin{eqnarray}
		&&a^1_\alpha(\phi,\psi)=\int_{\Omega_{H}}\nabla \phi \cdot\nabla \bar{\psi} dx+2\pi\sum_{n\in\mathbb{Z}}\sqrt{n^2+1} (\phi_{n}^{+}\bar\psi_{n}^{+}+\phi_{n}^{-}\bar{\psi}_{n}^{-}),\label{weaka1}\\
		&&    a^2_{\alpha}(\phi,\psi)=\int_{\Omega_{H}}\big(2i\alpha\frac{\partial \phi}{\partial x_{1}}\bar\psi+(k^{2}\gamma(x)-\alpha^{2})\phi\bar{\psi}\big)dx,\label{weaka2}\\
		&&a^3_\alpha(\phi,\psi)=\sum_{n\in\mathbb{Z}}(i\Lambda\beta_{n}+2\pi\sqrt{n^2+1})(\phi^{+}_{n}\bar{\psi}^{+}_{n}+\phi^{-}_{n}\bar{\psi}^{-}_{n}).\label{weaka3}
	\end{eqnarray}
	Clearly, $a_\alpha(\phi,\psi)=a^1_\alpha(\phi,\psi)-a^2_\alpha(\phi,\psi)-a^3_{\alpha}(\phi,\psi)$ and $a^{1}_{\alpha}(\phi,\psi)=(\phi,\psi)_*$. Next, using integration by parts, we have
    \begin{align*}
	a^2_{\alpha}(\phi,\psi)&=\int_{\Omega_{H}}\big(2i\alpha\frac{\partial \phi}{\partial x_{1}}\bar\psi+(k^{2}\gamma(x)-\alpha^{2})\phi\bar{\psi}\big)dx\\&=\int_{\Omega_{H}}\big(-2i\alpha\frac{\partial \bar\psi}{\partial x_{1}}\phi+(k^{2}\gamma(x)-\alpha^{2})\phi\bar{\psi}\big)dx.    
	\end{align*}
  By the compact embedding of $H^1_{per}(\Omega)\rightarrow L_{per}^2(\Omega_{H})$, we first obtain the embedding operator $E_{1}$, which is a compact operator, and consider $\hat{\phi}=E_{1}\phi$. Thus, for $\hat{\phi}$, we define the linear functional $L_{\hat{\phi}}(\psi)=a_{\alpha}^2(\phi,\psi)$. It is a bounded linear functional on $H^1_{per}(\Omega_{H})$ with respect to the norm $\Vert \cdot\Vert _{*}$. Thus, by the Riesz representation theorem, there exists a unique $\phi^{*}\in H^1_{per}(\Omega_{H})$ such that
	\begin{equation*}
		L_{\hat{\phi}}(\psi)=(\phi^{*},\psi)_{*}, \forall \psi\in H^{1}_{per}(\Omega_{H}).
	\end{equation*}
	Then, we are led to define $E^{*}(\alpha):L^2_{per}(\Omega_{H})\rightarrow H^{1}_{per}(\Omega_{H})$ given by $E^{*}(\alpha)\hat{\phi}=\phi^{*}$. This is a bounded operator since
	$|a^{2}_{\alpha}(\phi,\psi)|\leq C\Vert \hat\phi\Vert _{L^2(\Omega_H)}\Vert \psi\Vert _{*}$. So we define $E(\alpha)=E^{*}(\alpha)E_1$, and it is a compact operator $H^1_{per}(\Omega_{H})\rightarrow H^{1}_{per}(\Omega_{H})$.
	
	For $a^{3}_{\alpha}(\phi,\psi)$, since
	\begin{align*}
	    	&|i\Lambda\beta_{n}+2\pi\sqrt{n^2+1}|=
		\Lambda\frac{|(\frac{2\pi}{\Lambda})^2+k^2-\alpha^2-\frac{4\pi\alpha}{\Lambda}n|}{|-i\sqrt{k^2-(\frac{2\pi}{\Lambda}n+\alpha)^2}+  \frac{2\pi}{\Lambda}\sqrt{n^2+1}|} \\
        &\qquad=\Lambda\frac{|\big((\frac{2\pi}{\Lambda})^2+k^2-\alpha^2\big)\frac1n-\frac{4\pi\alpha}{\Lambda}|}{|-i\sqrt{(\frac kn) ^2-(\frac{2\pi}{\Lambda}+\frac{\alpha}{n})^2}+  \frac{2\pi}{\Lambda}\sqrt{1+\frac{1}{n^2}}|}
       \to \Lambda |\alpha| \quad (|n|\to \infty),
	\end{align*}
    there is a positive constant $C_0$ such that 
	$|i\Lambda\beta_{n}+2\pi\sqrt{n^2+1}|\leq C_0, \forall n\in \mathbb{Z}$.  We further consider the truncation of $a^{3}_{\alpha}(\phi,\psi)$. For $N\in \mathbb{Z}^+$, 
	\begin{eqnarray*}
		a^3_{\alpha,N}(\phi,\psi)=
		\sum_{|n|<N}(i\Lambda\beta_{n}+2\pi\sqrt{n^2+1})(\phi^{+}_{n}\bar{\psi}^{+}_{n}+\phi^{-}_{n}\bar{\psi}^{-}_{n}).
	\end{eqnarray*}
	It is clear that
	\begin{equation*}
		|a^3_{\alpha,N}(\phi,\psi)|\leq C \Vert \phi\Vert _{*}\Vert \psi\Vert _{*}.
	\end{equation*}
	By Riesz theory, there is a finite-rank operator $L_{N}(\alpha)\in\mathcal{L}(H^1_{per}(\Omega_{H}))$, such that 
	\begin{equation*}
		a^3_{\alpha,N}(\phi,\psi)=(L_{N}(\alpha)\phi,\psi)_{*}.
	\end{equation*}
	Since
	\begin{eqnarray*}
		&&|a^3_{\alpha,N}(\phi,\psi)-a^3_{\alpha}(\phi,\psi)|=|\sum_{|n|\geq N}(i\Lambda\beta_{n}+2\pi\sqrt{n^2+1})(\phi^{+}_{n}\bar{\psi}^{+}_{n}+\phi^{-}_{n}\bar{\psi}^{-}_{n})|\\
		&&\qquad \le \sum_{|n|\geq N}\frac{|i\Lambda\beta_{n}+2\pi\sqrt{n^2+1}|}{2\pi\sqrt{1+n^2}}2\pi\sqrt{1+n^2}|\phi^{+}_{n}\bar{\psi}^{+}_{n}+\phi^{-}_{n}\bar{\psi}^{-}_{n}|\\
		&&\qquad\leq \frac{C}{N}\Vert \phi\Vert _{*}\Vert \psi\Vert _{*},
	\end{eqnarray*}
	and the limit of compact operators is also compact, it follows that there is a compact operator $L(\alpha)\in\mathcal{L}(H^1_{per}(\Omega_{H}))$, such that
	\begin{equation*}
		a_{\alpha}^{3}(\phi,\psi)=(L(\alpha)\phi,\psi)_{*}.
	\end{equation*} 
    Finally, letting $K(\alpha)=E(\alpha)+L(\alpha)$, we obtain the compact operator as asserted.
\end{proof}
%\section{Properties of the propagative wave numbers}\label{sect3}
\subsection{Properties of the propagative wave numbers}
To further study the properties of the propagative wave numbers, we extend $D(\alpha)$ to a neighborhood of $[-\frac{\pi}{\Lambda},\frac{\pi}{\Lambda}]$ in $\mathbb{C}$. Here, to ensure the well-posedness and good analytical properties of the problem, it is common in the literature on quasi-periodic scattering to exclude certain cutoff values \cite{Kirsch1} in order to avoid pathological solutions to the linear equations. The cut-off values are defined by 
\begin{equation*}
	\mathcal{A}(k)=\lbrace \alpha\in [-\frac{\pi}{\Lambda},\frac{\pi}{\Lambda}]|k=|\alpha+\frac{2\pi}{\Lambda} m|,\,\textit{for some m}\in\mathbb{Z}\rbrace.
\end{equation*}
We follow the same approach and adopt the following assumption.
\begin{assume}\label{keyasump}
	The cut-off values are not propagative wave numbers, at $\hat{\alpha}\in \mathcal{A}(k)$, $\dim \operatorname{Ker}(D(\hat{\alpha}))$=0.
\end{assume}
Under this assumption, the following properties of the propagative wave numbers hold and have been proved in \cite{Kirsch2}.
\begin{lemma}\label{finiteness}
	Under Assumption \ref{keyasump}, we have
	\begin{enumerate}
		\item There exists an open set $\mathcal{W}\subset \mathbb{C}$, with $[-\frac{\pi}{\Lambda},\frac{\pi}{\Lambda}]\setminus \mathcal{A}(k)\subset \mathcal{W}$ that consists of at most three connected components, such that the operator $D(\alpha)$ depends analytically and has a finite dimensional kernel at $\alpha\in \mathcal{W}$.
		\item There exist at most finitely many propagative wave numbers of $D(\alpha)$ on $[-\frac{\pi}{\Lambda},\frac{\pi}{\Lambda}]$. Thus we denote by $P(k)$ the finite set of all propagative wave numbers on $[-\frac{\pi}{\Lambda},\frac{\pi}{\Lambda}]$. 
	\end{enumerate}
\end{lemma}

Let $B_{\delta}(\alpha)$ be the disk in the complex plane centered at $\alpha$ with radius $\delta$. By Lemma \ref{finiteness}, we further have the following result.
\begin{corollary}\label{analyprop}
Under Assumption 1, for each $\alpha_{j}\in P(k)$ with $j\in\mathcal{J}$ (where $\mathcal{J}$ is an index set), there exists $\delta_{j}\in (0, k)$ such that $\overline{B_{\delta_{j}}(\alpha_{j})} \cap \mathcal{A}(k)=\emptyset$. The operator $D(\alpha)$ is holomorphic and has a finite dimensional kernel for $\alpha\in \overline{B_{\delta_{j}}(\alpha_{j})}$. Moreover, $D(\alpha)$ is invertible for all $\alpha\in \overline{B_{\delta_{j}}(\alpha_{j})}\setminus \{\alpha_{j}\}$.
\end{corollary}
%%%%%%%%%%%%%%%%%%%%%%%%%%%%%%%%%%%%%%%%%%%%%%%%%%%%%%%
\section{Truncation by perfectly matched layer}\label{sect3}

\subsection{The PML formulation}
Since the major nonlinearity in the calculation of propagative wave numbers due to the Rayleigh expansion
is reflected by the DtN operator, we consider the introduction of perfectly matched layer to truncate the unbounded domain and arrive at a polynomial eigenvalue problem. The perfectly matched layer technique, see \cite{ChenWu} for instance, under the case of quasi-periodic scattering is formally a complex coordinate stretching in the $x_{2}$-direction. Namely, along the $x_{2}$-direction, we introduce two PML layers with length $\delta$ outside $\Omega_{H}$. The density of the PML is given by $s(x_{2})=s_{1}(x_{2})+is_{2}(x_{2})$, which satisfies that
\begin{equation*}
	s_{1},s_{2}\in C(\mathbb{R}),\quad s_{1}\geq 1,s_{2}\geq 0,\textit{ such that }s(x_{2})=1, -H\leq x_{2}\leq H .
\end{equation*}
Thus, the PML layers are given by 
\begin{eqnarray*}
	&        \Omega_{+}^{PML}=\lbrace(x_{1},x_{2}):-\frac{\Lambda}{2}<x_{1}<\frac{\Lambda}{2},\textit{ and }H\leq x_{2}\leq H+\delta\rbrace,\\
	&        \Omega_{-}^{PML}=\lbrace(x_{1},x_{2}):-\frac{\Lambda}{2}<x_{1}<\frac{\Lambda}{2},\textit{ and }-H-\delta\leq x_{2}\leq -H\rbrace,
\end{eqnarray*}
Now we introduce the PML differential operator $\mathcal{L}$ via
\begin{equation*}
	\mathcal{L}:=\frac{\partial}{\partial x_{1}}s(x_{2})\frac{\partial}{\partial x_{1}}+\frac{\partial}{\partial x_{2}}\frac{1}{s(x_{2})}\frac{\partial}{\partial x_{2}}+s(x_{2})k^{2}\gamma(x).
\end{equation*}
Let $\tilde{\Omega}_{H}=\Omega_{H}\cup\Omega^{PML}_{+}\cup\Omega^{PML}_{-}$, $\Gamma_1=\lbrace(x_{1},x_{2}):-\frac{\Lambda}{2}<x_{1}<\frac{\Lambda}{2}, x_{2}=H+\delta\rbrace$ and $\Gamma_2=\lbrace(x_{1},x_{2}):-\frac{\Lambda}{2}<x_{1}<\frac{\Lambda}{2}, x_{2}=-H-\delta\rbrace$, we obtain the PML equation,
\begin{equation}
	\mathcal{L}u=0 \, \, in \, \, \Omega.\label{PMLde}
\end{equation}
Further, we denote by $X(\tilde{\Omega}_{H})$ the following space
\begin{equation*}
	X(\tilde{\Omega}_{H}):=\lbrace w \in H_{\alpha}^{1}(\tilde{\Omega}_{H})|w=0 \text{ on }\Gamma_{1}\cup \Gamma_{2} \rbrace.
\end{equation*}
The weak formulation for the propagative wave numbers problem with PML truncation thus reads as follows: Find $\alpha \in \mathbb{C}$ and $0\neq u\in X(\tilde{\Omega}_{H})$ such that for all $v\in X(\tilde{\Omega}_{H})$, 
\begin{equation}
	a_{\tilde{\Omega}_{H}}(u,v)=\int_{\tilde{\Omega}_{H}}\big(s(x_{2})\frac{\partial u}{\partial x_{1}} \cdot\frac{\partial \bar{v}}{\partial x_{1}}+\frac{1}{s(x_{2})}\frac{\partial u}{\partial x_{2}} \cdot\frac{\partial \bar{v}}{\partial x_{2}} - s(x_{2})k^{2}\gamma(x)u\bar{v}\big)dx=0.\label{IMPplm}
\end{equation}
Next, we will introduce an explicit weak formulation for the propagative wave numbers problem. Let $\phi\in X_{per}(\tilde{\Omega}_{H})$, where $X_{per}(\tilde{\Omega}_{H}):=\{w\in H^1_{per}(\tilde{\Omega}_{H})| w=0 \text{ on }\Gamma_1\cup\Gamma_2\}$, such that if $u=\phi e^{i\alpha x_{1}}$ for $\phi \in X_{per}(\tilde{\Omega}_H)$, then $u\in X(\tilde\Omega_H)$. Substituting $u$ into the PML equation \eqref{PMLde}, we have the following
differential equation for $\phi$ in $\tilde{\Omega}_{H}$:
\begin{equation}
	\frac{\partial}{\partial x_{1}}s(x_{2})\frac{\partial}{\partial x_{1}}\phi+\frac{\partial}{\partial x_{2}}\frac{1}{s(x_{2})}\frac{\partial}{\partial x_{2}}\phi+2i\alpha s(x_{2}) \frac{\partial \phi}{\partial x_{1}}+s(x_{2})(k^{2}\gamma-\alpha^2)\phi=0.\label{phiPML0}
\end{equation}
We now introduce the following sesquilinear form
\begin{eqnarray*}
	&&a_{\alpha,\tilde{\Omega}_{H}}(\phi,\psi)=\int_{\tilde{\Omega}_{H}}\big(s(x_{2})\frac{\partial \phi}{\partial x_{1}} \cdot\frac{\partial \bar{\psi}}{\partial x_{1}} +\frac{1}{s(x_{2})}\frac{\partial \phi}{\partial x_{2}} \cdot\frac{\partial \bar{\psi}}{\partial x_{2}} \nonumber\\
	&&\qquad\qquad-2i\alpha s(x_{2})\frac{\partial \phi}{\partial x_{1}}\bar\psi-s(x_{2})
	(k^{2}\gamma(x)-\alpha^{2})\phi\bar{\psi}\big)dx,
	\text{ for }\phi,\psi\in X_{per}(\tilde\Omega_H).
\end{eqnarray*}
The corresponding propagative wave number problem then reads: find $\alpha\in\mathbb{C}$, such that there is $\phi\in X_{per}(\tilde{\Omega}_{H})$, $\phi\neq 0$, which satisfies:
\begin{equation}
	a_{\alpha,\tilde{\Omega}_{H}}(\phi,\psi)=0,\forall \psi\in X_{per}(\tilde{\Omega}_{H}).\label{Dpblm}
\end{equation}
This is simply a quadratic eigenvalue problem. Thus, the original non-polynomial eigenvalue problem has been reduced to a polynomial one. To analyze the perturbation of the eigenvalues induced by the PML truncation, we further consider the PML equation in the region $\tilde{\Omega}_{H}\setminus\Omega_{H}$. Since $\gamma(x)=1$, the equation \eqref{phiPML0} for $x \in \tilde{\Omega}_{H}\setminus \Omega_{H}$ becomes
\begin{equation}
	\frac{\partial}{\partial x_{1}}s(x_{2})\frac{\partial}{\partial x_{1}}\phi+\frac{\partial}{\partial x_{2}}\frac{1}{s(x_{2})}\frac{\partial}{\partial x_{2}}\phi+2i\alpha s(x_{2}) \frac{\partial \phi}{\partial x_{1}}+s(x_{2})(k^{2}-\alpha^2)\phi=0.\label{phiPML}
\end{equation}
 Since $\phi$ is periodic in the $x_{1}$ direction, we consider its Fourier series expansion, $\phi(x_{1},x_{2})=\sum\phi_{n}(x_{2})e^{i\frac{2n\pi}{\Lambda}x_{1}}$, where $\phi_n(x_2)=\frac{1}{\Lambda}\int^{\Lambda/2}_{-\Lambda/2}\phi(x_1,x_2)e^{i\frac{2n\pi}{\Lambda}x_{1}}dx_1$. Substituting this expansion into \eqref{phiPML}, we obtain for $x\in \Omega^{PML}_+\cup\Omega^{PML}_-$ that 
\begin{equation*}
	\sum_{n\in\mathbb{Z}}(\frac{\partial}{\partial x_{2}}\frac{1}{s(x_{2})}\frac{\partial}{\partial x_{2}}\phi_{n}(x_{2})+s(x_{2})(k-(\alpha+\frac{2\pi n}{\Lambda}))^2\phi_{n}(x_{2}))e^{i\frac{2n\pi}{\Lambda}x_{1}}=0.
\end{equation*}
Thus $\phi(x_{1},x_{2})$ for $x\in \Omega^{PML}_{+}\cup \Omega^{PML}_{-}$ can be written as, for some $A_{n}^\pm,B_{n}^\pm$, 
\begin{eqnarray*}
	\phi(x_{1},x_{2})=\left\{\begin{array}{ll}\sum\limits_{n\in\mathbb{Z}}(A_{n}^{+}e^{i\beta_{n}\int_{x_{2}}^{\delta+ H}s(\tau)d\tau}+B_{n}^{+}e^{-i\beta_{n}\int_{x_{2}}^{\delta+ H}s(\tau)d\tau})e^{i\frac{2n\pi}{\Lambda}x_{1}},&x_2>H,\\
		\sum\limits_{n\in\mathbb{Z}}(A_{n}^{-}e^{ i\beta_{n}\int^{x_{2}}_{-\delta-H}s(\tau)d\tau}+B_{n}^{-}e^{- i\beta_{n}\int^{x_{2}}_{-\delta- H}s(\tau)d\tau})e^{i\frac{2n\pi}{\Lambda}x_{1}},&x_2<-H.
	\end{array}\right.
\end{eqnarray*}
We also denote $\sigma^{+}=\int_{H}^{H+\delta}s(x_{2})dx_{2}$, $\sigma^{-}=\int_{-H-\delta}^{-H}s(x_{2})dx_{2}$, $|\sigma|=min\{|\sigma^{+}|,|\sigma^{-}|\}$. To determine $A_{n}^\pm,B_{n}^\pm$, we note that $\phi$ has well defined Fourier coefficients $\phi^{\pm}_{n},n\in\mathbb{Z}$ corresponding to its trace on $\Gamma_{H}^{\pm}$.
Furthermore, since $\phi(x_{1},x_{2})=0,x_{2}=\pm(\delta+ H)$, we obtain the following equations for $n\in\mathbb{Z}$
\begin{eqnarray*}	A_{n}^{\pm}e^{i\beta_{n}\sigma^{\pm}}+B_{n}^{\pm}e^{-i\beta_n\sigma^{\pm}}&=&\phi_{n}^{\pm},\\
	A_{n}^{\pm}+B_{n}^{\pm}&=&0.
\end{eqnarray*}
It follows that
\begin{equation*}
	A_{n}^{\pm}=-B_{n}^{\pm}=\phi_{n}^{\pm}/(e^{i\beta_{n}\sigma^{\pm}}-e^{-i\beta_{n}\sigma^{\pm}}).
\end{equation*}
Thus
\begin{eqnarray*}
	\phi(x_{1},x_{2})=\left\{\begin{array}{ll}\sum_{n\in\mathbb{Z}}\frac{e^{i\beta_{n}\int_{x_{2}}^{\delta+ H}s(\tau)d\tau}-e^{-i\beta_{n}\int_{x_{2}}^{\delta+ H}s(\tau)d\tau}}{e^{i\beta_{n}\sigma^{+}}-e^{-i\beta_{n}\sigma^{+}}}\phi^{+}_{n}e^{i\frac{2n\pi}{\Lambda}x_{1}}, &x\in\Omega^{PML}_{+},\\
		\sum_{n\in\mathbb{Z}}\frac{e^{ i\beta_{n}\int^{x_{2}}_{-\delta-H}s(\tau)d\tau}-e^{- i\beta_{n}\int^{x_{2}}_{-\delta-H}s(\tau)d\tau}}{e^{i\beta_{n}\sigma^{-}}-e^{-i\beta_{n}\sigma^{-}}}\phi_{n}^{-}e^{i\frac{2n\pi}{\Lambda}x_{1}}, &x\in \Omega^{PML}_{-}.
	\end{array}\right.
\end{eqnarray*}
Now we consider the derivative $\frac{\partial \phi}{\partial x_{2}}$ on $\Gamma_{\pm H}$, and introduce the following  PML DtN mapping $T^{\pm}_{\alpha,PML}: H^{\frac{1}{2}}(\Gamma_{\pm_{H}})\rightarrow H^{-\frac{1}{2}}(\Gamma_{\pm_{H}})$, notice that $\coth(\tau)=\frac{e^{\tau}+e^{-\tau}}{e^{\tau}-e^{-\tau}},\tau\in\mathbb{C}$, 
\begin{equation}
	T^{\pm}_{\alpha,PML}\phi:=\sum_{n\in\mathbb{Z}}i\beta_{n}\coth(-i\beta_{n}\sigma^{\pm})\phi_{n}^{\pm}e^{i\frac{2\pi n}{\Lambda}x_{1}}. \label{dtPML} 
\end{equation}
     Using Green's formula in $\tilde{\Omega}_{H}\setminus\Omega_{H}$, we find that for every eigenvalue $\alpha$ of \eqref{Dpblm} with a corresponding eigenvector $\phi$, and for all $\psi\in X_{per}(\tilde{\Omega}_{H})$,
\begin{eqnarray*}
	&&a_{\alpha,\tilde{\Omega}_{H}}(\phi  ,\psi )=\int_{\Omega_{H}}\big(\nabla \phi \cdot\nabla \bar{\psi} -2i\alpha\frac{\partial \phi}{\partial x_{1}}\bar\psi-(k^{2}\gamma(x)-\alpha^{2})\phi\bar{\psi}\big)dx\\
	&&\qquad\qquad-\int_{\Gamma_{H}}T_{\alpha,PML}^{+}\phi\bar{\psi}ds-\int_{\Gamma_{-H}}T_{\alpha,PML}^{-}\phi\bar{\psi}ds.
\end{eqnarray*}
Thus we are led to define, for all $\phi,\psi\in H^{1}_{per}(\Omega_{H})$,
\begin{eqnarray*}
	a_{\alpha,PML}(\phi ,\psi)=&\int_{\Omega_{H}}\big(\nabla \phi \cdot\nabla \bar{\psi} -2i\alpha\frac{\partial \phi}{\partial x_{1}}\bar\psi-(k^{2}\gamma(x)-\alpha^{2})\phi\bar{\psi}\big)dx\\
	&-\int_{\Gamma_{H}}T_{\alpha,PML}^{+}\phi\bar{\psi}ds-\int_{\Gamma_{-H}}T_{\alpha,PML}^{-}\phi\bar{\psi}ds.
\end{eqnarray*}
This leads to the PML eigenvalue problem: find $\alpha \in \mathbb{C}$ and $\phi\in H^1_{per}(\Omega_{H})$, $\phi\neq 0$, such that
\begin{equation}
	a_{\alpha,PML}(\phi,\psi)=0,\forall \psi\in H^1_{per}(\Omega_{H}).\label{PMLpblm}
\end{equation}
We have the following theorem.
\begin{theorem}
	    Let $\alpha$ be an eigenvalue of problem \eqref{PMLpblm} with a corresponding eigenfunction $\phi\in H^1_{per}(\Omega_{H})$. Then $\alpha$ is also an eigenvalue of problem \eqref{Dpblm}, and there exists a corresponding eigenfunction $\hat{\phi}\in X_{per}(\tilde{\Omega}_{H})$ such that $\phi = \hat{\phi}|_{\Omega_H}$.
\end{theorem}
\begin{proof}
	$\forall \phi,\psi\in H^1_{per}(\Omega_{H})$, for $|x_{2}|>H$, we consider the following function
	\begin{equation*}
		\xi^{\pm}_{n}(x_{2})=e^{\pm i\beta_{n}\int_{x_{2}}^{\pm(\delta+ H)}s(\tau)d\tau}-e^{\mp i\beta_{n}\int_{x_{2}}^{\pm(\delta+ H)}s(\tau)d\tau}.
	\end{equation*}
	Thus we define the following extension operator
	\begin{equation*}
		E^{\pm}(\alpha)\phi(x_{1},x_{2}) = \sum_{n\in\mathbb{Z}}\frac{\xi^{\pm}_{n}(x_{2})}{\xi^{\pm}_{n}(\pm H)}\phi_{n}^{\pm}e^{i\frac{2\pi}{\Lambda}nx_{1}}.
	\end{equation*}
	It follows that  $E^{\pm}(\alpha)\phi(x_{1},\pm (H+ \delta) )=0$, and we define $\hat{\phi}\in X_{per}(\tilde{\Omega}_{H})$ such that
	\begin{eqnarray*}
		\hat{\phi}(x_{1},x_{2})=\left\{\begin{array}{ll} E^{\pm}(\alpha)\phi(x_{1},x_{2}), &\text{if } \pm x_{2}\in (H,H+\delta),\\
			\phi(x_1,x_2), & \text{ Otherwise }.
		\end{array}\right.
	\end{eqnarray*}
	A direct computation verifies that $\hat{\phi}$ is an eigenvector of \eqref{Dpblm}.
\end{proof}

\subsection{Analysis of the PML DtN operator}%\label{sect5}
Recalling the PML DtN mapping given in \eqref{dtPML}, 
%\begin{equation}
%T^{\pm}_{\alpha,PML}\phi=\sum_{n\in\mathbb{Z}}i\beta_{n}\coth(-i\beta_{n}\sigma^{\pm})\phi_{n}^{\pm}e^{inx_{1}}. \label{dtPML} 
%\end{equation}
to discuss the perturbation of propagative wave numbers, we are led to investigate
\begin{equation*}
	\coth(-i\beta_{j}\sigma^{\pm})-1=\frac{2}{\exp(-2i\beta_{j}\sigma^{\pm})-1}.
\end{equation*}
Now, for a general $\sigma\in\mathbb{C}$, we give special emphasis on the estimate of $h(z)$, which is defined by,
\begin{equation*}
	h(z):=\exp(-2i\sqrt{k^2-z^2}\sigma).
\end{equation*}
In the following sections, we take a fixed $M>3$, and define
\begin{equation}\label{deltaPML}
	\delta_{PML}=\min\lbrace\frac{\delta_{j}}{M}:j\in\mathcal{J}\rbrace.
\end{equation}
where $\delta_j, j\in\mathcal{J}$ are defined in Corollary \ref{analyprop}. 
Furthermore, given all $\alpha_{j}\in P(k)$, we introduce the critical region $\mathcal{V}$ and its periodic extension $\mathcal{V}_{inn}$ by
\begin{equation*}
	\mathcal{V}=\cup_{j\in\mathcal{J}}\overline{B_{\delta_{PML}}(\alpha_{j})},\mathcal{V}_{inn}=\cup_{n\in\mathbb{Z}}(\mathcal{V}+\frac{2\pi n}{\Lambda}).
\end{equation*}
Now we can prove the following lemma.
\begin{lemma}\label{lemma1}
	Under Assumption \ref{keyasump}, let
	$\tau\in (\frac{\arctan(\frac{1}{M-1})}{2},\frac{\pi-\arctan(\frac{2(3M-1)}{((3M-1)(M-1)-1)})}{2})$,  $\sigma = |\sigma|e^{i\tau}$,  there exits $\gamma_0>0$, independent of $|\sigma|$, $\delta_{PML}$, such that
	\begin{equation}
		|h(z)|\geq \exp(\gamma_0\sqrt{\delta_{PML}}|\sigma|\sqrt{|\Re(z)|+k}),\label{Exdec}
	\end{equation}
	for all $z\in \mathcal{V}_{inn}$.
\end{lemma}

\begin{proof}
	Letting $\sigma=|\sigma|e^{i\tau}$, $\sqrt{k^{2}-z^{2}}=r_{z}e^{i\theta}$. We note that the choice of the branch cut gives $\theta\in (-\frac{\pi}{4},\frac{3\pi}{4})$. It follows that
	\begin{equation}
		\Re (-2i\sqrt{k^2-z^2}\sigma)=2r_{z}|\sigma|\sin(\theta+\tau).\label{exsin}
	\end{equation}
	Then, to prove the estimate \eqref{Exdec}, it suffices to prove the existence of $\gamma_0 >0$ independent of $z\in \mathcal{V}_{inn}$ and an estimate on the lower bound of $r_{z}$ independent of $|\sigma|$. To start with, we write $z = a+ ib$. By the definition of $\delta_{PML}$, we have $|b|\le \delta_{PML}$ for all $z\in \mathcal{V}_{inn}$. Moreover, it follows immediately from the definition of $\delta_{PML}$ that $M\delta_{PML}<k$. In addition, we have
    \begin{equation*}
    |k - |a|| \ge |k - |\alpha_j|| - |a - \alpha_j| \ge (M-1)\delta_{PML}.
    \end{equation*}
    Consequently, either $|a| \le k - (M-1)\delta_{PML}$ or $|a| \ge k + (M-1)\delta_{PML}$. We treat these two cases separately.
    
        \textbf{Case(1).}  $|a| \le k-(M-1)\delta_{PML}$. In this case, we have
        	\begin{equation*}
			r_{z}^2=|k^{2}-z^{2}|=|k^{2}-a^2+b^2-2abi|\geq |k^{2}|-|a^2|\geq (|a|+k)(M-1)\delta_{PML}.
		\end{equation*}    
		It also shows that $\Re (k^2-z^2)>0$, then $\theta\in (-\frac{\pi}{4},\frac{\pi}{4})$. Now we consider 
		\begin{equation*}
			|\tan 2\theta|=|\frac{2ab}{k^2-a^2+b^2}|\leq |\frac{2a\delta_{PML}}{(|a|+k)(M-1)\delta_{PML}}|=\frac{2|a|}{(M-1)(|a|+k)}\le \frac{1}{M-1},
		\end{equation*}
		thus, $ \theta\in (-\frac{\arctan(\frac{1}{M-1})}{2},\frac{\arctan{\frac{1}{M-1}}}{2})$, for given $\tau\in (\frac{\arctan{\frac{1}{M-1}}}{2},\pi-\frac{\arctan{\frac{1}{M-1}}}{2})$, let $\gamma_{1}=\min\{\sin(\tau-\frac{\arctan{\frac{1}{M-1}}}{2}),\sin(\tau+\frac{\arctan{\frac{1}{M-1}}}{2})\}$, then $\sin(\theta+\tau)\ge\gamma_{1}>0$. 
      
        \textbf{Case(2).} $|a|\ge k+(M-1)\delta_{PML}$. We have
        \begin{eqnarray*}
			&&   |a|+k=|a|-k+2k\ge (3M-1)\delta_{PML},\\
			&&   r_{z}^2=|z^{2}-k^{2}|=|a^2-b^2+2abi-k^2|\geq |a^2-k^2-b^2|\\
			&&~  \geq (|a|+k)(M-1)\delta_{PML}-\delta_{PML}^2
			\geq (|a|+k)((M-1)-1/(3M-1))\delta_{PML}.
		\end{eqnarray*}
		Moreover, since $\Re (k^2-z^2)<0$, we have $\theta\in(\frac{\pi}{4},\frac{3\pi}{4})$. Furthermore, since
		\begin{eqnarray*}
			|\tan 2\theta|=|\frac{2ab}{k^2-a^2+b^2}|\leq |\frac{2a\delta_{PML}}{(|a|+k)((M-1)-1/(3M-1))\delta_{PML}}|\\
			=\frac{2(3M-1)|a|}{((3M-1)(M-1)-1)(|a|+k)}\le\frac{2(3M-1)}{((3M-1)(M-1)-1)}.
		\end{eqnarray*}
		Thus, by letting $\zeta=\arctan(\frac{2(3M-1)}{((3M-1)(M-1)-1)})$, we have 
%		\begin{equation*}
		$	\theta\in(\frac{\pi-\zeta}{2},\frac{\pi+\zeta}{2})$.
	%	\end{equation*}
		As $|\tau|<\frac{\pi-\zeta}{2}$ is given, we take $\gamma_{2}=\min\{\cos(\tau-\frac{\zeta}{2}),\cos(\tau+\frac{\zeta}{2})\}$, then $\sin(\theta+\tau)\ge \gamma_{2}>0$.

		By taking $\tau\in (\frac{\arctan(\frac{1}{M-1})}{2},\frac{\pi-\arctan(\frac{2(3M-1)}{((3M-1)(M-1)-1)})}{2})$, we have the desired estimate \eqref{Exdec}. 
	\end{proof}
	\begin{remark}\label{tauassum}
		We note that according to Lemma \ref{lemma1}, for $M>0$ large enough, the region for valid $\tau$ expands to $(0,\frac{\pi}{2})$, so the technical assumption given by \cite{Zhang1} is not necessary for our analysis. We further remark that taking $\tau=\frac{\pi}{4}$ is sufficient; therefore, we automatically require that $\tau$ meets the requirement of Lemma \ref{lemma1} in the following sections.
	\end{remark}
	\begin{theorem}\label{expdecay}
		Under Assumption \ref{keyasump},  for $|\sigma|$ large enough, there is a compact operator $K_{\sigma}(\alpha)$ from $H^{1}_{per}(\Omega_{H})$ into itself such that
		\begin{equation*}
			a_{\alpha,PML}(\phi,\psi)=(D_{\sigma}(\alpha)\phi,\psi)_{*}=((I-K_{\sigma}(\alpha))\phi,\psi)_{*}.
		\end{equation*}
Here $I$ denotes the identity operator. Moreover, one can find a constant $C_{exp}>0$ such that
		\begin{equation*}
			\Vert D(\alpha)-D_{\sigma}(\alpha)\Vert _{*}\leq C_{exp} e^{-\gamma_0\sqrt{k\delta_{PML}}|\sigma|}
		\end{equation*}
		for all $\alpha\in \mathcal{V}$. Here $\gamma_0$ is defined in Lemma \ref{lemma1} and $D(\alpha)$ is defined in Theorem \ref{Thm:1}.
	\end{theorem}
	%% We include a proof here for reference
	\begin{proof}
		For any $\alpha_{j}\in P(k)$, we consider $\alpha\in B_{\delta_{PML}}(\alpha_{j})$. Once again we have that
		\begin{eqnarray*}
			&&a_{\alpha,PML}(\phi,\psi)=\int_{\Omega_{H}}\nabla \phi \cdot\nabla \bar{\psi} dx+2\pi\sum_{n\in\mathbb{Z}}\sqrt{n^2+1}(\phi_{n}^{+}\bar\psi_{n}^{+}+\phi_{n}^{-}\bar\psi_{n}^{-})\\
			&&\qquad-\int_{\Omega_{H}}2i\alpha\frac{\partial \phi}{\partial x_{1}}\bar\psi dx- \int_{\Omega_{H}}(k^{2}\gamma(x)-\alpha^2)\phi\bar\psi dx\\
			&&\qquad-\sum_{n\in\mathbb{Z}}(i\Lambda\beta_{n}\coth(-i\beta_{n}\sigma^{+})+2\pi\sqrt{n^2+1})\phi^{+}_{n}\bar{\psi}^{+}_{n}\\
			&&\qquad-\sum_{n\in\mathbb{Z}}(i\Lambda\beta_{n}\coth(-i\beta_{n}\sigma^{-})+2\pi\sqrt{n^2+1})\phi^{-}_{n}\bar{\psi}^{-}_{n}.
		\end{eqnarray*} 
			By the definitions \eqref{weaka1}-\eqref{weaka2}, we find that 
		$$a_{\alpha,PML}(\phi,\psi)=a^1_{\alpha}(\phi,\psi)-a^2_{\alpha}(\phi,\psi)-a^3_{\alpha,PML}(\phi,\psi),$$
		where
		\begin{eqnarray*}
					&&a^3_{\alpha,PML}(\phi,\psi)=\sum_{n\in\mathbb{Z}}(i\Lambda\beta_{n}\coth(-i\beta_{n}\sigma^{+})+2\pi\sqrt{n^2+1})\phi^{+}_{n}\bar{\psi}^{+}_{n}\\
		&&\qquad\qquad+\sum_{n\in\mathbb{Z}}(i\Lambda\beta_{n}\coth(-i\beta_{n}\sigma^{+})+2\pi\sqrt{n^2+1})\phi^{-}_{n}\bar{\psi}^{-}_{n}.
		\end{eqnarray*}
		Then we have 
			$a^{1}_{\alpha}(\phi,\psi)=(\phi,\psi)_*$, and  $a^{2}_{\alpha}(\phi,\psi)=(E(\alpha)\phi,\psi)_*,$
		where $E(\alpha)$ is the same as the one in the proof of Theorem \ref{Thm:1}.
		For $a^{3}_{\alpha,PML}(\phi,\psi)$, since $\alpha_{n}\in \mathcal{V}_{inn}$ for $n\in \mathbb{Z}$, we use Lemma \ref{lemma1} to obtain that,
		\begin{equation*}
			|\beta_{n}(\coth(-i\beta_{n}\sigma^{\pm})-1)|\leq C|\beta_{n}|\exp(-\gamma_0\sqrt{\delta_{PML}}|\sigma|\sqrt{|(Re(\alpha_{n}))|+k}).
		\end{equation*}
		which is $o(\frac{1}{n})$ as $n\to\infty$. Thus, taking any $N \in \mathbb{Z}^+$, 
		\begin{eqnarray*}
			&&\quad\quad|\sum_{|n|>N}(i\Lambda\beta_{n}\coth(-i\beta_{n}\sigma^{\pm})+2\pi\sqrt{n^2+1})\phi^{\pm}_{n}\bar{\psi}^{\pm}_{n}|\\
			&&~~=\Lambda|\sum_{|n|>N}(i\beta_{n}\coth(-i\beta_{n}\sigma^{\pm})-i\beta_{n}+i\beta_{n}+\frac{2\pi}{\Lambda}\sqrt{n^2+1})\phi^{\pm}_{n}\bar{\psi}^{\pm}_{n}|\\
			&&~~\leq \frac{C}{N} \Vert \phi\Vert _{*}\Vert \psi\Vert _{*}.
		\end{eqnarray*}
		By a similar argument as in Theorem \ref{Thm:1}, there exists a compact operator $L_{\sigma}(\alpha)$ such that
		\begin{equation*}
			a_{\alpha,PML}^{3}(\phi,\psi)=(L_{\sigma}(\alpha)\phi,\psi)_{*}.
		\end{equation*} 
		Then letting $K_\sigma(\alpha)=E(\alpha)+L_\sigma(\alpha)$, we obtain the compact operator as asserted. 
		Finally, since 
		\begin{eqnarray*}
			|(D_{\sigma}(\alpha)\phi,\psi)_{*}-(D(\alpha)\phi,\psi)_{*}|&\leq& |\Lambda\sum_{n\in\mathbb{Z}}\beta_{n}(\coth(-i\beta_{n}\sigma^{\pm})-1)\phi^\pm_{n}\bar{\psi}^\pm_{n}|\\
			&\leq&C e^{-\gamma_0|\sigma|\sqrt{k\delta_{PML}}}|\sum_{n\in\mathbb{Z}}\beta_{n}\phi^\pm_{n}\bar{\psi}^\pm_{n}|\\
			&\leq&C_{exp}e^{-\gamma_0|\sigma|\sqrt{k\delta_{PML}}}\Vert \phi\Vert _{*}\Vert \psi\Vert _{*}.
		\end{eqnarray*}
		And the assertion follows by the definition of operator norm.
	\end{proof}
	\section{Multiplicity of the propagative wave numbers} 
	\label{sect4}
	Since we are interested in the property of propagative wave numbers inside the holomorphic region, let us introduce the concepts concerning the multiplicity of the propagative wave number $\alpha_{j}\in P(k)$ following \cite{Ammari}.
	
    Let $\alpha_0$ be a fixed complex value. Denote by $\mathcal{U}(\alpha_{0})$ the set of all operator-valued functions with values in $\mathcal{L}(H^1_{per}(\Omega_{H}))$ which are holomorphic in some neighborhood of $\alpha_0$, except possibly at $\alpha_{0}$. It follows that $D(\alpha)\in \mathcal{U}(\alpha_{j})$ by Corollary \ref{analyprop}.  As the problem of propagative wave numbers is a nonlinear eigenvalue problem, we need to adapt the following definition on the nonlinear eigenvalues of holomorphic operators.

\begin{definition}\cite{Ammari}
The point $\alpha_{0}$ is called an \textbf{eigenvalue} of $S(\alpha)\in \mathcal{U}(\alpha_{0})$, if there exists a vector-valued function $\phi(\alpha)$ with values in $H^{1}_{per}(\Omega_{H})$, such that
\begin{enumerate}
\item $\phi({\alpha})$ is holomorphic at $\alpha_{0}$, and $\phi(\alpha_{0})\neq 0$,
\item $S(\alpha)\phi(\alpha)$ is holomorphic at $\alpha_{0}$ and vanishes at this point.
\end{enumerate} 
Here, $\phi(\alpha)$ is called a root function of $S(\alpha)$, associated with $\alpha_{0}$. The vector $\phi_{0}=\phi(\alpha_{0})$ is called an \textbf{eigenvector} of $S(\alpha)$. 
The closure of the linear set of eigenvectors corresponding to $\alpha_{0}$ is denoted by $\operatorname{Ker}(S(\alpha_{0}))$. Then there exists a number $m(\phi)\geq 1$, and a vector-valued function $\psi(\alpha)$ with values in $H^{1}_{per}(\Omega_{H})$, holomorphic at $\alpha_{0}$ such that
\begin{equation*}
S(\alpha)\phi(\alpha)=(\alpha-\alpha_{0})^{m(\phi)}\psi(\alpha),~\psi(\alpha_{0})\neq 0.
\end{equation*}
 We define the rank of $\phi_{0}$, denoted by $\operatorname{rank}(\phi_{0})$, to be 
 \begin{equation*}
     \operatorname{rank}(\phi_{0})=\max\{m(\phi)|\phi_{0}=\phi(\alpha_{0}),\phi(\alpha) \textit{ is a root function of }S(\alpha)\}.
 \end{equation*} 
\end{definition}
	Thus, every propagative wave number $\alpha_{j}\in P(k)$ is a eigenvalue of $D(\alpha)$, and $n_j:=\dim(\operatorname{Ker}(D(\alpha_{j})))$ is finite by Lemma \ref{finiteness}. To analyze the properties of $D(\alpha)$ and $D^{-1}(\alpha)$ near $\alpha_j$, we introduce the following definition.

\begin{definition}\cite{Ammari2010}\label{normalpoint}
Assume that $\alpha_{0}$ is a pole of $S(\alpha) \in \mathcal{U}(\alpha_0)$, and the Laurent series of $S(\alpha)$ at $\alpha_{0}$ has the following form:
    \begin{equation*}
        S(\alpha)=\sum_{j\geq -m}^{\infty}(\alpha-\alpha_{0})^{j}S_{j}.
    \end{equation*}
    If the operators $\{S_{n}\}_{n=-m}^{-1}$ are all of finite rank, we call $S(\alpha)$ finitely meromorphic at $\alpha_{0}$. If the operator $S_{0}$ is Fredholm (of index zero), we say $S(\alpha)$ is of Fredholm type (of index zero) at the point $\alpha_{0}$. If $S(\alpha)$ is invertible and holomorphic at $\alpha_{0}$, we call $\alpha_{0}$ a regular point of $S(\alpha)$. If every point of $V\subset\mathbb{C}$ is a regular point of $S(\alpha)$, we call $S(\alpha)$ regular in $V$. If $S(\alpha)$ is finitely meromorphic, of Fredholm type at $\alpha_{0}$ and is regular in a punctured neighborhood of $\alpha_{0}$, we call $\alpha_{0}$ a normal point of $S(\alpha)$.
\end{definition}

Thus, $\alpha_{j}\in P(k)$ are normal points of $D(\alpha)$. By Theorem 1.8 and Corollary 1.9 of \cite{Ammari}, the operator $D(\alpha)$ admits a decomposition of the form  

\begin{equation}\label{decom1}
    D(\alpha)=E_j(\alpha)\Bigl(P_{j,0}+\sum_{i=1}^{k_j} (\alpha-\alpha_j)^{m_{j,i}}P_{j,i}\Bigr)F_j(\alpha).
\end{equation}

Here $E_j(\alpha)$ and $F_j(\alpha)$ are regular at $\alpha_j$. $\{P_{j,i}\}_{i=0}^{k_j}$ are mutually disjoint projections with $P_{j,1},\cdots,P_{j,n}$ being rank‑one operators, and they satisfy $\sum_{i=0}^{k_j} P_{j,i}=I$. Let $\mathbb{N}_0$ denote the set of non‑negative integers. The Riesz number of a bounded linear operator is defined as follows.   

\begin{definition}\cite{Lay,Martensen}
    Let $\mathcal{H}$ be a Banach space and $\mathcal{T}\in \mathcal{L}(X)$. The ascent and descent of $\mathcal{T}$ are
$$
\operatorname{asc}(\mathcal{T}):=\inf\{n\in\mathbb N_0:\operatorname{Ker} (\mathcal{T}^{n})=\operatorname{Ker}( \mathcal{T}^{n+1})\},$$
$$
\operatorname{dsc}(\mathcal{T}):=\inf\{n\in\mathbb N_0:\operatorname{Ran}(\mathcal{T}^{n})=\operatorname{Ran}(\mathcal{T}^{n+1})\},
$$
with value $\infty$ if the corresponding set is empty. If $\operatorname{asc}(\mathcal{T})=\operatorname{dsc}(\mathcal{T})<\infty$; their common value $r(\mathcal{T})$ is called the Riesz number of $\mathcal{T}$. Moreover, we have
$$ \mathcal{H}= \operatorname{Ker}(\mathcal{T}^{r(\mathcal{T})}) \oplus \operatorname{Ran}(\mathcal{T}^{r(\mathcal{T})})$$
\end{definition}
From Theorem \ref{Thm:1}, we have that $D(\alpha_j)$ can be written as the sum of the identity operator and a compact operator. Consequently, $\operatorname{asc}(D(\alpha_j)) = \operatorname{dsc}(D(\alpha_j)) < \infty$. For a more precise description of the decomposition \eqref{decom1}, we prove the following result for eigenvalues $\alpha_j \in  P(k)$: 
      \begin{theorem}\label{rieszno}
		Under Assumption \ref{keyasump}, the Riesz number of $D(\alpha_{j}),\alpha_j\in P(k)$ is one. 
	\end{theorem}
	\begin{proof}
		For a fixed propagative wave number $\alpha_{j}\in P(k)$, let $\phi \in \operatorname{Ker}(D(\alpha_J))$. Since $a_{\alpha_j}(\phi,\phi)=0$, we have
		\begin{eqnarray*}
			\int_{\Omega_{H}}\bigl(|\nabla \phi|^2 +2\alpha_{j}\Im(\bar\phi\frac{\partial \phi}{\partial x_{1}})-(k^{2}\gamma(x)-\alpha_{j}^{2})|\phi|^2 \bigr)dx=\sum_{n\in\mathbb{Z}}i\Lambda\beta_{n}(|\phi^{+}_{n}|^2+|\phi^{-}_{n}|^2). 
		\end{eqnarray*}
		Thus, taking the imaginary part of both sides, we obtain that
		\begin{equation}\label{nonprop}
			\sum_{n\in B_{\alpha_{j}}}\beta_{n}(|\phi^{+}_{n}|^2+|\phi^{-}_{n}|^2)=0.
		\end{equation}
		It follows that $\phi^{+}_{n}=\phi^{-}_{n}=0,n\in B_{\alpha_{j}}$. Now, letting $\xi\in \operatorname{Ker}(D^2(\alpha_{j}))$, and $\eta=D(\alpha_j)\xi\in \operatorname{Ker}(D(\alpha_{j}))$. Since $\eta^{+}_{n}=\eta^{-}_{n}=0,n\in B_{\alpha_{j}}$, we have
		\begin{equation*}
			(\eta,\eta)_{*}=a_{\alpha_{j}}(\xi,\eta)=\overline{a_{\alpha_{j}}(\eta,\xi)}=0.
		\end{equation*}
		Namely, $\eta=0$, thus $\xi\in \operatorname{Ker}(D(\alpha_{j}))$. So  $\operatorname{Ker}(D^2(\alpha_{j}))\subset \operatorname{Ker}(D(\alpha_{j}))$, which proves the assertion.
	\end{proof}
	We remark that by equation \eqref{nonprop}, the following corollary holds.
	\begin{corollary}\label{corollary2}
		Under Assumption \ref{keyasump}, for all $\alpha_{j}\in P(k)$, the corresponding guided modes only have evanescent modes. 
	\end{corollary}
 Theorem \ref{rieszno} gives the decomposition $H_{per}^1(\Omega_H)=\operatorname{Ker}(D(\alpha_{j})) \oplus \operatorname{Ran}(D(\alpha_{j}))$.
    Let $P_j^{\operatorname{Ker}}$ be the projection onto $\operatorname{Ker}(D(\alpha_{j}))$ along $\operatorname{Ran}(D(\alpha_{j}))$, and set $P_j^{\operatorname{Ran}}=I-P_j^{\operatorname{Ker}}$. By a straightforward adaptation of the proof of Theorem 1.8 in \cite{Ammari}, the following decomposition can be readily established, we omit its proof.

\begin{theorem}\label{thm:decom_exact}
For each $j \in \mathcal{J}$, there exists $\delta_{j}^* \in (0,\delta_j)$ such that, on $B_{\delta_j^*}(\alpha_j)$, the operator $D(\alpha)$ admits the decomposition
\begin{equation}\label{decom_exact}
D(\alpha)=E_j(\alpha)\Bigl(P_{j}^{\operatorname{Ran}}+\sum_{i=1}^{n_j} (\alpha-\alpha_j)^{r_{j,i}}P_{j,i}^{\operatorname{Ker}}\Bigr)F_j(\alpha).
\end{equation}
Here, $E_j(\alpha)$ and $F_j(\alpha)$ are holomorphic and invertible on $B_{\delta_j^*}(\alpha_j)$. The operators $P_{j}^{\operatorname{Ran}},P_{j,1}^{\operatorname{Ker}},P_{j,2}^{\operatorname{Ker}},$ $\cdots,P_{j,n_j}^{\operatorname{Ker}}$ are mutually disjoint projections. Each $P_{j,i}^{\operatorname{Ker}}: H_{per}^1(\Omega_H)\to \operatorname{Ker}(D(\alpha_j))$ is a rank‑one projection, and they satisfy $\sum_{i=1}^{n_j}P_{j,i}^{\operatorname{Ker}}=P_{j}^{\operatorname{Ker}}$. The positive integers $r_{j,i}$ satisfy $r_{j,1} \ge r_{j,2} \ge \cdots \ge r_{j,n_j} \ge 1$, where $n_j = \dim(\operatorname{Ker}(D(\alpha_j)))$. 
\end{theorem}
By Corollary 3.1 of \cite{Gohberg}, we can determine the possible ranks of the eigenvectors of $D(\alpha_j)$.

\begin{corollary}\label{rankvalue}
    For any eigenvalue $\alpha_{j}\in P(k)$, the kernel $\operatorname{Ker}(D(\alpha_j))$ admits a basis $\{\phi_{j,i}\}_{i=1}^{n_j}$ such that $\operatorname{rank}(\phi_{j,i})=r_{j,i}$ for $i=1,2,\dots,n_j$. Moreover, for every nonzero $\phi \in \operatorname{Ker}(D(\alpha_j))$, the rank satisfies $\operatorname{rank}(\phi)\in \{r_{j,i}\}_{i=1}^{n_j}$.
\end{corollary}

In the domain $B_{\delta_j^*}(\alpha_j) \setminus \{ \alpha_j \}$, taking the inverse of \eqref{decom_exact} yields
\begin{equation}\label{Dinv-decom}
     D^{-1}(\alpha)=F^{-1}_j(\alpha)\Bigl(P_{j}^{\operatorname{Ran}}+\sum_{i=1}^{n_j} (\alpha-\alpha_j)^{-r_{j,i}}P_{j,i}^{\operatorname{Ker}}\Bigr)E_j^{-1}(\alpha). 
\end{equation}

Since $(\alpha-\alpha_j)^{r_{j,1}}D^{-1}(\alpha)$ is holomorphic at $\alpha_j$ and non‑zero at $\alpha=\alpha_j$, we know that $D^{-1}(\alpha)$ admits a Laurent expansion on $B_{\delta_j^*}(\alpha_j)\setminus\{\alpha_j\}$ of the form
\begin{equation}\label{Dinv-laurant}
    D^{-1}(\alpha)=\sum_{m=-r_{j,1}}^{\infty}(\alpha-\alpha_j)^mD_{j,m}.
\end{equation}

Because $D^{-1}(\alpha)$ is holomorphic in $B_{\delta_j}(\alpha_j)\setminus\{\alpha_j\}$, this expansion holds for every $\alpha \in B_{\delta_j}(\alpha_j)\setminus\{\alpha_j\}$. The following lemma describes the properties of the negative‑power coefficients $\{D_{j,m}\}_{m=-r_{j,1}}^{-1}$ in this expansion.

\begin{lemma}\label{lemma: Dinv-negative-rank}
    For each $m=-r_{j,1},\cdots,-1$, $D_{j,m}$ is a finite‑rank operator, and
    $$ 
    \dim\bigl(\operatorname{Ran}(D_{j,m})\bigr) \le \sum_{i=1}^{n_j} r_{j,i}.
    $$
\end{lemma}

\begin{proof}
    In $B_{\delta_j^*}(\alpha_j)$, the operators $F_j^{-1}(\alpha)$ and $E_j^{-1}(\alpha)$ are holomorphic. We expand them as
    $$
    F_j^{-1}(\alpha)=\sum_{m=0}^{\infty}(\alpha-\alpha_j)^mF_{j,m}, \qquad 
    E_j^{-1}(\alpha)=\sum_{m=0}^{\infty}(\alpha-\alpha_j)^mE_{j,m} .
    $$
    The series converge in operator norm in $B_{\delta_j^*}(\alpha_j)$. For $m=-r_{j,1},\dots,-1$, comparing the coefficients on both sides of \eqref{Dinv-decom} gives
    $$
    D_{j,m}= \sum_{i=1}^{n_j} 
             \sum_{\substack{m_e,m_f \in \mathbb{N}_0\\ m_e+m_f-r_{j,i}=m}} 
             F_{j,m_f}P_{j,i}^{\operatorname{Ker}}E_{j,m_e}.
    $$
    Every term in the sum on the right‐hand side has rank at most one. The total number of terms equals 
    $\sum_{i=1}^{n_j}\max\{0,\,m+r_{j,i}+1\}$. Hence $D_{j,m}$ is a finite‑rank operator, and
    $$
    \dim\bigl(\operatorname{Ran}(D_{j,m})\bigr) 
        \le \sum_{i=1}^{n_j}\max\{0,\,m+r_{j,i}+1\}
        \le \sum_{i=1}^{n_j} r_{j,i}.
    $$\end{proof}
We now introduce the notion of multiplicity for the subsequent discussion.  
\begin{definition}\cite{Ammari2010}
  For a fixed $\alpha_{0} \in \mathbb{C}$, let $S(\alpha) \in \mathcal{U}(\alpha_0)$ be an operator-valued function that admits $\alpha_0$ as an eigenvalue. We assume that $\dim \operatorname{Ker}(S(\alpha_{0}))<\infty$, $\operatorname{rank}(\phi_{0})<\infty,\, \forall 0 \neq\phi_{0}\in \operatorname{Ker}(S(\alpha_{0}))$. A canonical system of eigenvectors $\phi_{0}^{m},m=1,\cdots,\dim(\operatorname{Ker}(S(\alpha_0)))$ of $\operatorname{Ker}(S(\alpha))$ associated to $\alpha_{0}$ means that $\operatorname{rank}(\phi_{0}^{m})$ is the largest in the direct complement in $\operatorname{Ker}(S(\alpha_{0}))$ of the space spanned by $\{\phi_{0}^{n}\}_{n=1}^{m-1}$. For any canonical system of eigenvectors in
$\operatorname{Ker}(S(\alpha_{0}))$, the null-multiplicity of $S(\alpha)$ at $\alpha_{0}$ is denoted by
 \begin{equation}\label{nullmulti}
    N(S(\alpha_{0})):=\sum_{m=1}^{\dim(\operatorname{Ker}(S(\alpha_0)))}\operatorname{rank}(\phi_{0}^{m}).
    \end{equation}
    If $\alpha_{0}$ is not an eigenvalue of $S(\alpha)$, we let $N(S(\alpha_{0}))=0$. If $S^{-1}(\alpha)$ is well-defined, and $S^{-1}(\alpha)\in \mathcal{U}(\alpha_{0})$, we define the multiplicity of $S(\alpha)$ at $\alpha_{0}$ by
    \begin{equation*}   M(S(\alpha_{0}))=N(S(\alpha_{0}))-N(S^{-1}(\alpha_{0})).
    \end{equation*}
\end{definition}
\begin{definition}\cite{Ammari}\label{normality}
    For any simply connected region $\mathcal{V}\subset\mathbb{C}$ that have rectifiable boundary $\partial \mathcal{V}$, define
    \begin{eqnarray*}
        &&\mathcal{E}(\mathcal{V}):=\{S(\alpha)\in \mathcal{L}(H^{1}_{per}(\Omega_{H}))|\textit{$S(\alpha)$ is finitely meromorphic for $\alpha\in \mathcal{V}$}, \\
        && \qquad\qquad\textit{ is of Fredholm-type
       and is continuous on $\partial \mathcal{V}$.}\}.
    \end{eqnarray*}
    If $S(\alpha)\in \mathcal{E}(\mathcal{V})$ is invertible in $\bar{\mathcal{V}}$, except for a finite number of points of $\mathcal{V}$ which are normal points of $S(\alpha)$, then $S(\alpha)$ is called normal with respect to $\partial \mathcal{V}$. We denote $\hat{\alpha}_{i}, i=1, \cdots, n$ the normal points of $S(\alpha)$ in $\mathcal{V}$, we can define the multiplicity of $S(\alpha)$ with respect to $\partial \mathcal{V}$
    \begin{equation*}
        \mathcal{M}(S(\alpha);\partial \mathcal{V}):=\sum_{i=1}^{n}M(S(\hat{\alpha}_{i})).
    \end{equation*}
\end{definition}
For any $\alpha_{j}\in P(k)$,  the operator $D(\alpha)$ is normal with respect to  $\partial B_{\delta_{PML}}(\alpha_{j})$.  In addition, $\alpha_{j}$ is the only normal point of $D(\alpha)$ in $B_{\delta_{PML}}(\alpha_{j})$.  It is evident that the basis $\{\phi_{j,i}\}_{i=1}^{n_j}$ of $\operatorname{Ker}(D(\alpha_j))$ defined in Corollary~\ref{rankvalue} forms a canonical system of eigenvectors associated with the eigenvalue $\alpha_{j}$. Let $m_j=\mathcal{M}(D(\alpha);\partial B_{\delta_{PML}}(\alpha_{j}))$, we have
		\begin{equation}m_{j}=M(D(\alpha_{j}))=-M(D^{-1}(\alpha_j))=N(D(\alpha_{j}))=\sum_{i=1}^{n_{j}}r_{j,i}.\label{regmult}
		\end{equation}   

	\section{Analysis of the eigenvalue perturbation}\label{sect5}
	We also need some results from Gohberg and Sigal theory\cite{Gohberg, Krein, Ammari}, mainly the following generalized Rouche's theorem and generalized argument theorem.   
	\begin{lemma}\label{rouche}(Generalized Rouche's theorem) Assume that $V\subset \mathbb{C}$ is a simply connected region with rectifiable boundary and $S(\alpha)\in \mathcal{L}(H^1_{per}(\Omega_{H}))$ is normal with respect to $\partial V$, If the operator-valued function $W(\alpha)$ is finitely meromorphic in $V$, is continuous on $\partial V$ and satisfies 
			$$\Vert S^{-1}(\alpha)W(\alpha)\Vert _{*}<1, \alpha\in \partial V. $$
	It follows that $S(\alpha)+W(\alpha)$ is normal with respect to $\partial V$, and
		\begin{equation*}
			\mathcal{M}(S(\alpha);\partial V) = \mathcal{M}(S(\alpha)+W(\alpha);\partial V).
		\end{equation*}    
	\end{lemma}  
    For a proof of Lemma \ref{rouche}, we refer the reader to, e.g., \cite{Ammari}.
	\begin{lemma}\label{multiplicity}
		Under Assumptions \ref{keyasump}, for all $\alpha_{j}\in P(k)$ and for $|\sigma|$ large enough, $D_{\sigma}$ is normal with respect to $\partial B_{\delta_{PML}}(\alpha_{j})$. Furthermore,
		\begin{equation*}
			\mathcal{M}(D_{\sigma}(\alpha);\partial B_{\delta_{PML}}(\alpha_{j}))=m_{j}.
		\end{equation*}     
	\end{lemma}    
	\begin{proof}
		We denote 
			$$C_{\delta_{PML}}=\max_{j\in\mathcal{J}}\max_{\alpha\in \partial B_{\delta_{PML}}(\alpha_{j})} \{ \Vert D^{-1}(\alpha)\Vert _{*} \}, $$
		and let $\Delta_{\sigma}(\alpha)=D_\sigma(\alpha)-D(\alpha)$.   
		By the definition of $\delta_{PML}$ in \eqref{deltaPML}, $\Delta_{\sigma}(\alpha)$ is holomorphic in $B_{\delta_{PML}}(\alpha_{j})$, thus it is finitely meromorphic. Furthermore, it is continuous on $\partial B_{\delta_{PML}}(\alpha_{j})$. By Theorem \ref{expdecay}, $\Vert \Delta_{\sigma}(\alpha)\Vert _{*}$ is of exponential decay in $B_{\delta_{PML}}(\alpha_{j})$ with respect to $|\sigma|$. Taking $|\sigma|$ large enough, such that for all $\alpha\in\partial B_{\delta_{PML}}(\alpha_{j})$ and $j\in \mathcal{J}$, $\Vert D^{-1}(\alpha)\Delta_{\sigma}(\alpha)\Vert _{*}<C_{\delta_{PML}}\Vert \Delta_{\sigma}(\alpha)\Vert _{*}<1$. Since $D(\alpha)$ is normal with respect to $\partial B_{\delta_{PML}}(\alpha_{j})$. Using Lemma \ref{rouche} for $D(\alpha)$ and $\Delta_{\sigma}(\alpha)$, we obtain that  
		\begin{eqnarray*}
		&&	\mathcal{M}(D_\sigma(\alpha);\partial B_{\delta_{PML}}(\alpha_{j})) = \mathcal{M}(D(\alpha)+\Delta_{\sigma}(\alpha);\partial B_{\delta_{PML}}(\alpha_{j}))\\
		&&~~\qquad	=\mathcal{M}(D(\alpha);\partial B_{\delta_{PML}}(\alpha_{j}))=m_j.
		\end{eqnarray*} 
		This finishes the proof.    
        \end{proof}  
We further need the following definition from \cite{Ammari} of the trace of a finite rank operator, 
\begin{definition}\label{deftrace}
    Given a linear operator $F$ that has a finite rank in the Hilbert space $\mathcal{H}$, it has an inner product, $(\cdot,\cdot)_{\mathcal{H}}$ with norm $\|\cdot\|_{\mathcal{H}}$. Thus, there exists a finite-dimensional invariant subspace $\mathcal{C}$ of $F$ such that $F$ annihilates some direct complement of $\mathcal{C}$ in $\mathcal{H}$. The restriction of $F$ is denoted by $F|_{\mathcal{C}}$, the trace of $F|_{\mathcal{C}}$ is given by \cite{Kato1995}
    \begin{equation*}
trF=tr(F|_{\mathcal{C}}).
    \end{equation*}
  Taking an orthonormal basis $\{\phi_\mathcal{C}^j\}_{j=1}^{\dim(\mathcal{C})}$ of $\mathcal{C}$, we have
     \begin{equation}\label{fintrace}
    trF=\sum_{j=1}^{\dim(\mathcal{C})}(F\phi_{\mathcal{C}}^{j},\phi_{\mathcal{C}}^{j})_{\mathcal{H}}.
    \end{equation}
\end{definition}
Thus we can state the generalized argument theorem, see \cite{Gohberg} for a proof.
	\begin{lemma}\label{residue} (Generalized argument theorem) Assume that $S(\alpha)$ is normal with respect to $\partial V$. If $f(\alpha)$ is holomorphic in $V$, and is continuous in $\bar{V}$, we have 
		\begin{equation*}
			\frac{1}{2\pi i}tr\int_{\partial V}f(\alpha)S^{-1}(\alpha)\frac{d}{d\alpha}S(\alpha)d\alpha = \sum_{i=1}^{\kappa}M(S(\beta_{j}))f(\beta_{j}).
		\end{equation*}    Here    $\beta_{j}$, $j=1$, $\cdots$, $\kappa$ are poles or eigenvalues of $S(\alpha)$ in $V$.  
	\end{lemma} 
	Using the above generalized residue theorem, taking $S(\alpha)=D_{\sigma}(\alpha)$, $f(\alpha)=\alpha-\alpha_{j}$, for $m_{j}=1$, we have 
	\begin{equation}
		(\alpha^{\sigma}_{j}-\alpha_{j})=\frac{1}{2\pi i} tr\int_{\partial B_{\delta_{PML}}(\alpha_{j})}(\alpha-\alpha_{j})(D_{\sigma})^{-1}(\alpha)\frac{d}{d\alpha}D_{\sigma}(\alpha)d\alpha . \label{tracest}
	\end{equation}  
	
	If $m_{j}>1$, by Lemma \ref{multiplicity}, for $|\sigma|$ large enough, all the eigenvalues $D_{\sigma}(\alpha)$ in $B_{\delta_{PML}}(\alpha_{j})$ are denoted by $\alpha_{i,j}^{\sigma}$, $i=1$, $\cdots$, $n_{j}^{\sigma}$, we further have $\sum_{i=1}^{n_{j}^{\sigma}}M(D_{\sigma}(\alpha_{i,j}^{\sigma}))=m_{j}$, thus
	\begin{equation}
		\begin{aligned}
   &\sum_{i=1}^{n_{j}^{\sigma}}M(D_{\sigma}(\alpha^{\sigma}_{i,j}))(\alpha^{\sigma}_{i,j}-\alpha_{j})\\
   =&\frac{1}{2\pi i} tr\int_{\partial B_{\delta_{PML}}(\alpha_{j})}(\alpha-\alpha_{j})(D_{\sigma})^{-1}(\alpha)\frac{d}{d\alpha}D_{\sigma}(\alpha)d\alpha . \end{aligned}\label{multracest}
	\end{equation}  
	To complete the perturbation analysis, we need to give precise estimates on the right-hand side of \eqref{tracest}, \eqref{multracest}. For the computation of trace operators, we present the following lemma:
    \begin{lemma}\label{lemma:tracecal}
 Let $\mathcal{H}$ be a Hilbert space equipped with the inner product $(\cdot,\cdot)_{\mathcal{H}}$ and the induced norm $\|\cdot\|_{\mathcal{H}}$. For any bounded linear operator $\mathcal{T} \in \mathcal{L}(\mathcal{H})$ with $\mathcal{T}$ finite-rank, we have
\begin{equation}\label{trace-estimate}
    |\operatorname{tr} \mathcal{T}| \le 2\dim(\operatorname{Ran}(\mathcal{T}))  \|\mathcal{T}\|_{\mathcal{H}}, 
\end{equation}
here $\|\mathcal{T}\|_{\mathcal{H}}$ is the operator norm.
    \end{lemma}
    \begin{proof}
Since $\mathcal{T}$ is finite-rank, its adjoint $\mathcal{T}^*$ is also finite-rank, and $\operatorname{Ran}(\mathcal{T}^*)$ is the orthogonal complement of $\operatorname{Ker}(\mathcal{T})$, i.e., $\mathcal{H}=\operatorname{Ran}(\mathcal{T}^*) \oplus \operatorname{Ker}(\mathcal{T})$, with
$$
(x,y)_{\mathcal{H}}=0,\quad \forall x \in \operatorname{Ran}(\mathcal{T}^*), \, y \in \operatorname{Ker}(\mathcal{T}).
$$
Consider $\mathcal{C}=\operatorname{span}\{ \mathcal{T}^m x \mid x \in \operatorname{Ran}(\mathcal{T}^*), \, m \in \mathbb{N}_0 \}.$
It is easy to verify that $\operatorname{Ran}(\mathcal{T}|_\mathcal{C}) \subset \mathcal{C} \subset \operatorname{Ran}(\mathcal{T}^*) + \operatorname{Ran}(\mathcal{T})$, hence $\mathcal{C}$ is a finite-dimensional invariant subspace of $\mathcal{T}$, and $\dim (\mathcal{C}) \le 2\dim(\operatorname{Ran}(\mathcal{T}))$. Therefore, $\mathcal{C} \cap \operatorname{Ker}(\mathcal{T})$ is a finite-dimensional subspace of the closed subspace $\operatorname{Ker}(\mathcal{T})$ of $\mathcal{H}$, and thus it admits an orthogonal complement within $\operatorname{Ker}(\mathcal{T})$. That is, there exists a closed subspace $\mathcal{D}$ of $\mathcal{H}$ such that 
$$
\operatorname{Ker}(\mathcal{T}) = (\operatorname{Ker}(\mathcal{T}) \cap \mathcal{C}) \oplus \mathcal{D}
$$
and
$$
(x,y)_{\mathcal{H}} = 0, \quad \forall x \in (\operatorname{Ker}(\mathcal{T}) \cap \mathcal{C}), \, y \in \mathcal{D}.
$$
Consequently,
\begin{equation*}
\mathcal{H} = \operatorname{Ran}(\mathcal{T}^*) \oplus \operatorname{Ker}(\mathcal{T}) 
            = \operatorname{Ran}(\mathcal{T}^*) \oplus (\operatorname{Ker}(\mathcal{T}) \cap \mathcal{C}) \oplus \mathcal{D} 
            \subset \mathcal{C} + \mathcal{D} \subset \mathcal{H}.
\end{equation*}
Moreover,
$$
\mathcal{C} \cap \mathcal{D} = \mathcal{C} \cap (\mathcal{D} \cap \operatorname{Ker}(\mathcal{T})) = (\mathcal{C} \cap \operatorname{Ker}(\mathcal{T})) \cap \mathcal{D} = \{0\}.
$$
Hence $\mathcal{H} = \mathcal{C} \oplus \mathcal{D}$. Since $\mathcal{C}$ is a finite-dimensional invariant subspace of $\mathcal{T}$ and $\mathcal{T}$ vanishes on $\mathcal{D}$, we have $\operatorname{tr}\mathcal{T} = \operatorname{tr}(\mathcal{T}|_{\mathcal{C}})$. Taking an orthonormal basis $\{\phi_\mathcal{C}^i\}_{i=1}^{\dim(\mathcal{C})}$ of $\mathcal{C}$, we obtain
\begin{equation*}
|\operatorname{tr}\mathcal{T}| = \Bigl| \sum_{i=1}^{\dim(\mathcal{C})} (\mathcal{T}\phi_{\mathcal{C}}^{i},\phi_{\mathcal{C}}^{i})_{\mathcal{H}} \Bigr| \le \dim (\mathcal{C}) \, \|\mathcal{T}\|_{\mathcal{H}} \le 2\dim(\operatorname{Ran}(\mathcal{T}))\,\|\mathcal{T}\|_{\mathcal{H}}.
\end{equation*}
\end{proof}

  We present the following trace estimate. 
	\begin{lemma}\label{traceineq}  
		For all $p \in \mathbb{Z}^+$ and $\alpha_{j}\in P(k)$, if $\{S_i(\alpha)\}_{i=1}^p$ are holomorphic on $\overline{B_{\delta_{PML}}(\alpha_{j})}$, we define
$$
H(\alpha)=D^{-1}(\alpha)S_1(\alpha)D^{-1}(\alpha) S_2(\alpha) \cdots D^{-1}(\alpha)S_p(\alpha),
$$
then
\begin{equation*}
\left|\frac{1}{2\pi i}\operatorname{tr}\int_{\partial B_{\delta_{PML}}(\alpha_{j})}H(\alpha)\,d\alpha\right|
\leq C_1C_2^{\,p}\prod_{ i=1}^p\max_{\alpha \in \partial B_{\delta_{PML}}(\alpha_j)} \| S_i(\alpha) \|_*.
\end{equation*}
Here $C_1=C_1(m_j, \delta_{PML})=\frac12 m_j\delta_{PML}$ and $C_2=C_2(r_{j,1},C_{\delta_{PML}})=2^{r_{j,1}+2}C_{\delta_{PML}}$ are constants independent of $p$.
\end{lemma}

\begin{proof}
Recall the Laurent expansion \eqref{Dinv-laurant} and Lemma \ref{lemma: Dinv-negative-rank}. In $\overline{B_{\delta_{PML}}(\alpha_{j})}\setminus\{\alpha_j\}$ we have
$$
D^{-1}(\alpha)=\sum_{m=-r_{j,1}}^{\infty}(\alpha-\alpha_j)^mD_{j,m},
$$
and for $m=-r_{j,1},\dots,-1$, $
\dim\bigl(\operatorname{Ran}(D_{j,m})\bigr)\le \sum_{i=1}^{n_j} r_{j,i}=m_j.
$
For each $i=1,2,\dots,p$, expand the holomorphic function $S_i(\alpha)$ at $\alpha=\alpha_j$:
$$
S_i(\alpha)=\sum_{m=0}^{\infty}(\alpha-\alpha_j)^mS_{i,m},\qquad \alpha \in \overline{B_{\delta_{PML}}(\alpha_{j})}.
$$
Define the index set
$$
\mathcal{S}_{j,p}=\bigl\{(d_1,s_1,d_2,s_2,\cdots,d_p,s_p)\,\big|\,
\sum_{i=1}^p(d_i+s_i)=-1,\, s_i,d_i+r_{j,1} \in  \mathbb{N}_0\bigr\}.
$$
Then
$$
\frac{1}{2\pi i}\int_{\partial B_{\delta_{PML}}(\alpha_{j})}H(\alpha)\,d\alpha
= \sum_{(d_1,s_1,\cdots,d_p,s_p) \in \mathcal{S}_{j,p}}
D_{j,d_1}S_{1,s_1}D_{j,d_2}S_{2,s_2}\cdots D_{j,d_p}S_{p,s_p}.
$$
For a fixed $(d_1,s_1,\cdots,d_p,s_p) \in \mathcal{S}_{j,p}$, there exists some index $i_0 \in \{1,2,\cdots,p\}$ such that $d_{i_0}<0$. Hence
$$
\dim\bigl(\operatorname{Ran}(D_{j,d_1}S_{1,s_1}\cdots D_{j,d_p}S_{p,s_p})\bigr)
\le \dim\bigl(\operatorname{Ran}(D_{j,d_{i_0}})\bigr)
\le m_j.
$$
The cardinality of $\mathcal{S}_{j,p}$ is $\binom{pr_{j,1}+2p-2}{2p-1}$; therefore
$$
\dim\Bigl(\operatorname{Ran}\Bigl(\frac{1}{2\pi i}\int_{\partial B_{\delta_{PML}}(\alpha_{j})}H(\alpha)\,d\alpha\Bigr)\Bigr)
\le \binom{pr_{j,1}+2p-2}{2p-1}\,m_j
\le 2^{\,p(r_{j,1}+2)-2}\,m_j.
$$
Applying Lemma \ref{lemma:tracecal}, we obtain
\begin{align*}
&\quad \, \, \Bigl|\frac{1}{2\pi i}\operatorname{tr}\int_{\partial B_{\delta_{PML}}(\alpha_{j})}H(\alpha)\,d\alpha\Bigr| \\
&\le 2^{\,p(r_{j,1}+2)-1}\,m_j\,
\Bigl\|\frac{1}{2\pi i}\int_{\partial B_{\delta_{PML}}(\alpha_{j})}H(\alpha)\,d\alpha\Bigr\|_{*} \\
&\le 2^{\,p(r_{j,1}+2)-1}\,m_j\,
\frac{1}{2\pi}\int_{\partial B_{\delta_{PML}}(\alpha_{j})}\|H(\alpha)\|_{*}\,|d\alpha|  \\
&\le 2^{\,p(r_{j,1}+2)-1}\, m_j\,\delta_{PML}\,
\max_{\alpha \in \partial B_{\delta_{PML}}(\alpha_j)} \| H(\alpha) \|_* \\
&\le 2^{\,p(r_{j,1}+2)-1}\,m_j\,\delta_{PML}\,
C_{\delta_{PML}}^{\,p}\,
\prod_{i=1}^p \max_{\alpha \in \partial B_{\delta_{PML}}(\alpha_j)} \| S_i(\alpha) \|_*  \\
&= C_1\, C_2^{\,p}\,
\prod_{i=1}^p \max_{\alpha \in \partial B_{\delta_{PML}}(\alpha_j)} \| S_i(\alpha) \|_*,
\end{align*}
where $C_1=\frac12 m_j\delta_{PML}$ and $C_2=2^{r_{j,1}+2}C_{\delta_{PML}}$.
\end{proof}
Now we are ready to present the main perturbation theorem. 
	\begin{theorem}\label{singmult}(perturbation theorem for $m_{j}=1$) Under Assumptions \ref{keyasump}, for $m_{j}=1$ and large enough $|\sigma|$, we have the following perturbation estimate for $\alpha_{j}\in P(k)$,
		\begin{equation*}
			|\alpha^{\sigma}_{j}-\alpha_{j}|\leq C^*e^{-\gamma_{0}\sqrt{k\delta_{PML}}|\sigma|}.
	\end{equation*}   
    Here, $C^*=C^*(\delta_{PML},C_{\delta_{PML}},C_{exp})=8\delta_{PML}C_{\delta_{PML}}C_{exp}$. $\gamma_0$ and $C_{exp}$ are defined in Lemma \ref{lemma1} and Theorem \ref{expdecay}, respectively. 
\end{theorem}    
	\begin{proof}By Theorem \ref{expdecay}, we pick $|\sigma|$ large enough, such that $ \Vert \Delta_{\sigma}(\alpha)\Vert _{*}<\frac{1}{2C_2} < \frac{1}{2C_{\delta_{PML}}} $ for $\alpha \in \overline{ B_{\delta_{PML}}(\alpha_{j})}$, here $C_2=2^{r_{j,1}+2}C_{\delta_{PML}}$. Considering the following Neumann series expansion on $\partial B_{\delta_{PML}}(\alpha_{j})$: 
		\begin{eqnarray*}
			\nonumber(D_{\sigma})^{-1}(\alpha)=D^{-1}(\alpha)(I+\Delta_{\sigma}(\alpha)D^{-1}(\alpha))^{-1}=D^{-1}(\alpha)\sum_{p=0}^{\infty}(-\Delta_{\sigma}(\alpha)D^{-1}(\alpha))^{p},
		\end{eqnarray*}    
		 Using  \eqref{tracest} we have 
		\begin{eqnarray*}
	\alpha_{j}^{\sigma}-\alpha_{j}
            =\frac{1}{2\pi i}\sum_{p=0}^{\infty}tr\int_{\partial B_{\delta_{PML}}(\alpha_{j})}(\alpha-\alpha_{j})D^{-1}(\alpha)(-\Delta_{\sigma}(\alpha)D^{-1}(\alpha))^{p}\frac{d}{d\alpha}D_\sigma(\alpha)d\alpha. \\
        \end{eqnarray*}
Let  
\begin{align*}
    I_p &= \operatorname{tr}\int_{\partial B_{\delta_{PML}}(\alpha_{j})}
(\alpha-\alpha_{j})D^{-1}(\alpha)\bigl(-\Delta_{\sigma}(\alpha)D^{-1}(\alpha)\bigr)^{p}
\frac{d D(\alpha)}{d \alpha}\,d\alpha  \\
&=\operatorname{tr}\int_{\partial B_{\delta_{PML}}(\alpha_{j})}
(\alpha-\alpha_{j})\bigl(-D^{-1}(\alpha)\Delta_{\sigma}(\alpha)\bigr)^{p}D^{-1}(\alpha)
\frac{d D(\alpha)}{d \alpha}\,d\alpha.
\end{align*}
For $p\ge 1$, we apply Lemma \ref{traceineq} with $S_1(\alpha)=\cdots =S_{p}(\alpha)=-\Delta_{\sigma}(\alpha)$, $S_{p+1}(\alpha)=(\alpha-\alpha_j)\frac{d D(\alpha)}{d \alpha}$. Then we obtain 
$$ |I_p| \le 2 \pi C_1 C_2^{p+1} \max\limits_{\alpha\in\partial B_{\delta_{PML}}(\alpha_{j})}\Vert \Delta_{\sigma}(\alpha)\Vert _{*}^{p} \max\limits_{\alpha\in\partial B_{\delta_{PML}}(\alpha_{j})}\Vert(\alpha-\alpha_j)\frac{d D(\alpha)}{d \alpha} \Vert_*  .$$ 
Since $C_2\max\limits_{\alpha\in\partial B_{\delta_{PML}}(\alpha_{j})}\Vert \Delta_{\sigma}(\alpha)\Vert _{*} \le \frac12$, we have $\lim\limits_{p \to \infty} I_p =0$. Note that $$\frac{d}{d\alpha}D^{-1}(\alpha)=-D^{-1}(\alpha)\frac{d}{d\alpha}D(\alpha)D^{-1}(\alpha).$$  Therefore, for $p \ge 0$, integration by parts together with the properties of the operator trace \cite{Ammari2010} gives
\begin{align*}
  &\quad\,\, tr\int_{\partial B_{\delta_{PML}}(\alpha_{j})}\frac{1}{p+1}(-D^{-1}(\alpha)\Delta_{\sigma}(\alpha))^{p+1} d\alpha \\
  &=-tr\int_{\partial B_{\delta_{PML}}(\alpha_{j})}\frac{(\alpha-\alpha_{j})}{p+1}\frac{d}{d\alpha}(-D^{-1}(\alpha)\Delta_{\sigma}(\alpha))^{p+1} d\alpha\\
   &=   -tr\int_{\partial B_{\delta_{PML}}(\alpha_{j})}(\alpha-\alpha_{j})(-D^{-1}(\alpha)\Delta_{\sigma}(\alpha))^{p}\frac{d}{d\alpha}(-D^{-1}(\alpha)\Delta_{\sigma}(\alpha))d\alpha\\
    &=  tr\int_{\partial B_{\delta_{PML}}(\alpha_{j})}(\alpha-\alpha_{j})D^{-1}(\alpha)(-\Delta_{\sigma}(\alpha)D^{-1}(\alpha))^{p+1}\frac{d D(\alpha)}{d \alpha}d\alpha \\
     &\quad\,\,+tr\int_{\partial B_{\delta_{PML}}(\alpha_{j})}(\alpha-\alpha_{j})D^{-1}(\alpha)(-\Delta_{\sigma}(\alpha)D^{-1}(\alpha))^{p}\frac{d \Delta_{\sigma}(\alpha)}{d \alpha}d\alpha \\
     &=tr\int_{\partial B_{\delta_{PML}}(\alpha_{j})}(\alpha-\alpha_{j})D^{-1}(\alpha)(-\Delta_{\sigma}(\alpha)D^{-1}(\alpha))^{p}\frac{d}{d\alpha}D_\sigma(\alpha)d\alpha+I_{p+1}-I_p.
\end{align*}
Summing over all $p \ge 0$, we obtain
	\begin{align*}
	\alpha_{j}^{\sigma}-\alpha_{j}
            &=\frac{1}{2\pi i}\sum_{p=0}^{\infty}tr\int_{\partial B_{\delta_{PML}}(\alpha_{j})}(\alpha-\alpha_{j})D^{-1}(\alpha)(-\Delta_{\sigma}(\alpha)D^{-1}(\alpha))^{p}\frac{d}{d\alpha}D_\sigma(\alpha)d\alpha \\ 
            &=\frac{1}{2\pi i}I_0+\frac{1}{2\pi i}\sum_{p=1}^{\infty}tr\int_{\partial B_{\delta_{PML}}(\alpha_{j})}\frac{1}{p}(-D^{-1}(\alpha)\Delta_{\sigma}(\alpha))^{p} d\alpha. 
        \end{align*}

		Since $\alpha_{j}$  is the only propagative wave number of $D(\alpha)$ in    $B_{\delta_{PML}}(\alpha_{j})$, by Lemma \ref{residue} we can obtain
		\begin{equation*}
			\frac{1}{2\pi i}I_0=\frac{1}{2\pi i}tr\int_{\partial B_{\delta_{PML}}(\alpha_{j})}(\alpha-\alpha_{j})D^{-1}(\alpha)\frac{d}{d\alpha}(D(\alpha))d\alpha =0.
		\end{equation*}    
		By Theorem \ref{expdecay},
		\begin{equation*}
			\Vert \Delta_{\sigma}(\alpha)\Vert _{*}\leq C_{exp}\exp(-\gamma_{0}\sqrt{k\delta_{PML}}|\sigma|),
		\end{equation*}   
		Finally, since $C_2\max\limits_{\alpha\in\partial B_{\delta_{PML}}(\alpha_{j})}\Vert \Delta_{\sigma}(\alpha)\Vert _{*}<\frac12$, Lemma \ref{traceineq} gives 
		\begin{align*}
			|\alpha^{\sigma}_{j}-\alpha_{j}| & \le C_1\sum_{p=1}^{\infty} \bigl (C_2\max\limits_{\alpha\in\partial B_{\delta_{PML}}(\alpha_{j})}\Vert \Delta_{\sigma}(\alpha)\Vert _{*} \bigr)^p\\
            &\le 2C_1C_2\max\limits_{\alpha\in\partial B_{\delta_{PML}}(\alpha_{j})}\Vert \Delta_{\sigma} (\alpha)\Vert _{*} \\
            &\le 2C_1C_2C_{exp}\exp(-\gamma_{0}\sqrt{k\delta_{PML}}|\sigma|).
		\end{align*}
        In the case where $ m_j = 1 $, we have $ r_{j,1} = 1 $. Consequently, the constants $ C_1 $ and $ C_2 $ defined in Lemma \ref{traceineq} become $ C_1 = \frac{1}{2} \delta_{PML} $ and $ C_2 = 8 C_{\delta_{PML}} $. Therefore,
$$
|\alpha^{\sigma}_{j} - \alpha_{j}| \le 8 \delta_{PML} \, C_{\delta_{PML}} \, C_{exp} \, \exp( -\gamma_0 \sqrt{k \delta_{PML}} \, |\sigma| ).
$$
	\end{proof}    
	Using \eqref{multracest}, similar to the proof of Theorem \ref{singmult}, we can obtain the following result. 
	\begin{theorem}(For $m_{j}>1$)Under Assumptions \ref{keyasump}, for large enough $|\sigma|$, the following estimate holds for all $\alpha_{j}\in P(k)$
		\begin{equation*}
	|\sum_{i=1}^{n_{j}^{\sigma}}M(D_{\sigma}(\alpha^{\sigma}_{i,j}))(\alpha^{\sigma}_{i,j}-\alpha_{j})|\leq C^*e^{-\gamma_{0}\sqrt{k\delta_{PML}}|\sigma|} .
		\end{equation*}  
	\end{theorem}
    Here, $C^*=C^*(m_j,r_{j,1}\delta_{PML},C_{\delta_{PML}},C_{exp})=2^{r_{j,1}+2}m_j\delta_{PML}C_{\delta_{PML}}C_{exp}$. $\gamma_0$ and $C_{exp}$ are defined in Lemma \ref{lemma1} and Theorem \ref{expdecay}, respectively. 
    
	\section{Discretization of the PML problem}\label{sect6}  
	In this section, we first transform the PML problem into a block matrix representation, then carry out the finite element discretization, and finally obtain the final numerical format as well as its finite element discretization error.
	\subsection{Block matrix representation of the PML problem}
	For all $\phi, \psi \in H^1_{per}(\tilde{\Omega}_H)$, let  
	\begin{eqnarray*}
		&&a_{1}(\phi,\psi)=\int_{\tilde{\Omega}_H}s(x_{2})\frac{\partial \phi}{\partial x_{1}} \cdot\frac{\partial \bar{\psi}}{\partial x_{1}} +\frac{1}{s(x_{2})}\frac{\partial \phi}{\partial x_{2}} \cdot\frac{\partial \bar{\psi}}{\partial x_{2}} -k^2(\gamma(x)) s(x_{2})\phi\bar{\psi}dx, \\ 
		&&a_{2}(\phi,\psi)=-2i\int_{\tilde{\Omega}_H}s(x_{2})\frac{\partial \phi}{\partial x_{1}}\bar\psi dx, \\ 
		&&a_{3}(\phi,\psi)=\int_{\tilde{\Omega}_H}s(x_{2})\phi\bar{\psi}dx.
	\end{eqnarray*}    
    We can rewrite the problem \eqref{Dpblm} as: find nonvanishing $\phi\in X_{per}(\tilde{\Omega}_H)$ and $\alpha\in\mathbb{C}$ such that
	\begin{equation}
		a_{\alpha,\tilde{\Omega}_{H}}(\phi,\psi)=a_{1}(\phi,\psi)+\alpha a_{2}(\phi,\psi)+\alpha^2 a_{3}(\phi,\psi)=0, ~\forall\psi\in X_{per}(\tilde{\Omega}_H). \label{Dpblm1} 
	\end{equation}    
	According to the Riesz representation theorem, there exist linear operators $A_{\sigma,i}: X_{per}(\tilde{\Omega}_H) \rightarrow X_{per}(\tilde{\Omega}_H), \, i=1,2,3$, such that $a_i(\phi,\psi)=(A_{\sigma,i}\phi,\psi)_{H^1_{per}(\tilde\Omega_H)}$. To compute this quadratic eigenvalue problem, we introduce a shift of the eigenvalue as in \cite{Engstrom1}. Take $ t \in \mathbb{C} $ and define  
$$
\tilde{A}_{\sigma,1}=A_{\sigma,1}+tA_{\sigma,2}+t^2A_{\sigma,3},\quad 
\tilde{A}_{\sigma,2}=A_{\sigma,2}+2tA_{\sigma,3},\quad 
\tilde{A}_{\sigma,3}=A_{\sigma,3}.
$$
Then  
$$
A_{\sigma,1}+A_{\sigma,2}\alpha +A_{\sigma,3}\alpha^{2}
= \tilde{A}_{\sigma,1}+\tilde{A}_{\sigma,2}(\alpha-t) +\tilde{A}_{\sigma,3}(\alpha-t)^{2}.
$$
	Then, we can obtain an equivalent block matrix representation of \eqref{Dpblm1}: find nonvanishing $\phi,\eta\in X_{per}(\tilde{\Omega}_H)$ and $\alpha\in\mathbb{C}$ such that
	\begin{equation}
		\begin{bmatrix}
			(\tilde{A}_{\sigma,1}\phi,\psi)_{H^1_{per}(\tilde{\Omega}_H)}  \\ 
			(\eta,\xi)_{H^1_{per}(\tilde{\Omega}_H)}
		\end{bmatrix}
		=(\alpha-t)
		\begin{bmatrix}
			-(\tilde{A}_{\sigma,2}\phi  +\tilde{A}_{\sigma,3}\eta,\psi)_{H^1_{per}(\tilde{\Omega}_H)} \\ 
			(\phi,\xi)_{H^1_{per}(\tilde{\Omega}_H)}  
		\end{bmatrix} ,\, \forall \psi,\xi \in X_{per}(\tilde{\Omega}_H) .
		\label{blockeig}
	\end{equation}  
	Therefore, let
	\begin{equation*}
		\mathcal{M}_{\sigma}^{0}=
		\begin{bmatrix}
			\tilde{A}_{\sigma,1} & 0 \\ 
			0 & I
		\end{bmatrix},~
		\mathcal{M}_{\sigma}^{1}=
		\begin{bmatrix}
			-\tilde{A}_{\sigma,2} &  -\tilde{A}_{\sigma,3} \\ 
			I  & 0
		\end{bmatrix}.
	\end{equation*}    
	Here, $\mathcal{M}_{\sigma}^{i}$, $i=0,1$ are operators from $X_{per}(\tilde{\Omega}_H) \times X_{per}(\tilde{\Omega}_H) \rightarrow X_{per}(\tilde{\Omega}_H) \times X_{per}(\tilde{\Omega}_H)$. Then \eqref{blockeig} is equivalent to
	\begin{equation*}
(\mathcal{M}_{\sigma}^{0}\Phi,\Psi)_{1}=(\alpha-t)(\mathcal{M}_{\sigma}^{1}\Phi,\Psi)_{1}, \forall \Psi \in X_{per}(\tilde{\Omega}_H) \times X_{per}(\tilde{\Omega}_H).
	\end{equation*}
	Here $\Phi = (\phi, \eta), \phi, \eta \in X_{per}(\tilde{\Omega}_H)$, $\Psi = (\psi, \xi), \psi, \xi \in X_{per}(\tilde{\Omega}_H)$ and $(\Phi, \Psi)_{1} = (\phi, \psi)_{
		H^1_{per}(\tilde{\Omega}_H)} + (\eta, \xi)_{H^1_{per}(\tilde{\Omega}_H)}$. If $t \in \{ t \in \mathbb{C} | A_{\sigma,1}+tA_{\sigma,2}+t^2A_{\sigma,3} \textit{ is isomorphism.} \}$, then $\mathcal{M}_{\sigma}^{0}$ is an isomorphism. Then, we can define $\mathcal{L}_{\sigma} = (\mathcal{M}_{\sigma}^{0})^{-1}\mathcal{M}_{\sigma}^1$. Therefore, the generalized eigenvalue problem \eqref{Dpblm1} becomes a linear eigenvalue problem of finding the eigenvalues $\mu_{j}^{\sigma} = \frac{1}{\alpha_{j}^{\sigma}-t}$ of $\mathcal{L}_{\sigma}$.
	
	\subsection{Finite element discretization}
	Now we consider the finite element approximation of the block matrix equation \eqref{blockeig}. Let $M_{0}$ represent the matrix on the left side of the equation, and let $M_{1}$ represent the matrix on the right side. Take $\lbrace \mathcal{M}_{h}\rbrace$ as a regular triangulation of the domain $\tilde{\Omega}_H$, where $h>0$ is the longest side of all triangles. Furthermore, assume that any triangle $T$ must be completely contained in $\overline{\Omega_{+}^{PML}}$, $\overline{\Omega_{-}^{PML}}$, or $\overline{\Omega}$. The nodal basis functions corresponding to the triangulation $\{ \mathcal{M}_{h} \}$ are denoted as $\psi_{i}\in V_{h}(\tilde{\Omega}_H)$, $i=1$, $\cdots$, $N$.  We take $V_{h}(\tilde{\Omega}_H)\subset X_{per}(\tilde{\Omega}_H)$ as a space composed of first-order elements. In actual calculations, appropriate adjustments at the left and right boundaries can make space $V_{h}(\tilde{\Omega}_H)$ satisfy the periodic boundary conditions in the $x_{1}$ direction \cite{ChenWu}. Furthermore,
	\begin{equation*}
		\lim_{h\rightarrow 0}\inf_{u_{h}\in V_{h}(\tilde{\Omega}_H)}\Vert u-u_{h}\Vert _{H_{per}^{1}(\tilde{\Omega}_H)}=0,  \textit{ for all }         u\in X_{per}(\tilde{\Omega}_H).
	\end{equation*}
	Next, discretize the equation by taking $\phi_{h}=\sum_{i=1}^{N}\gamma_{i}\psi_{i}$, $\eta_{h}=\sum_{i=1}^{N}\xi_{i}\psi_{i}$ to represent the finite element interpolation functions corresponding to $\phi, \eta \in X_{per}(\tilde{\Omega}_H)$. Substituting them into \eqref{blockeig} gives the PML truncation propagative value problem under finite element discretization: find $\alpha \in \mathbb{C}$ such that there exist non-zero $\hat{\gamma}, \hat{\xi}$ satisfying,
	\begin{equation}\label{disblockdiag}
		\begin{bmatrix}
	A_{\sigma,1}^h+tA_{\sigma,2}^h+t^2A_{\sigma,3}^h & 0 \\ 
			0 & I^{h}
		\end{bmatrix}\begin{bmatrix}
			\hat{\gamma} \\ 
			\hat{\xi}
		\end{bmatrix}
		=(\alpha-t)
		\begin{bmatrix}
			-A_{\sigma,2}^h-2tA_{\sigma,3}^h & -A_{\sigma,3}^h \\ 
			I^{h} & 0
		\end{bmatrix}\begin{bmatrix}
			\hat{\gamma} \\ 
			\hat{\xi}
		\end{bmatrix}.
	\end{equation}

	Here the $(i,j)$ entries of matrix $A^h_{\sigma,k}$ are 
	$ a_{k}(\psi_{j},\psi_{i}),\,k=1,2,3$. $I^h$ is the identity matrix on the finite-dimensional space. In this way, we obtain the generalized eigenvalue problem corresponding to the PML problem.

	\subsection{Discretization error}   
	For each $\mu^{\sigma}$ (which is an eigenvalue of $L_{\sigma}$), take a contour $C_{{\mu}^{\sigma}}$ such that on $C_{{\mu}^{\sigma}}$, $L_{\sigma}$ is invertible, and within the bounded region enclosed by $C_{{\mu}^{\sigma}}$, $L_{\sigma}$ has only the eigenvalue $\mu^{\sigma}$. Similarly, $C_{\bar{\mu}^{\sigma}}$ can be taken. From this, we define two eigenvalue projections $E({\mu}^{\sigma}), E^{*}(\bar{\mu}^{\sigma})$ following the definition in \cite{Babuska, Boffi,Engstrom1}:
	\begin{eqnarray*}
		E(\mu^{\sigma})=\frac{1}{2\pi i}\int_{C_{\mu^{\sigma}}}(z-L_{\sigma})^{-1}dz, \\ 
		E^{*}(\bar{\mu}^{\sigma})=\frac{1}{2\pi i}\int_{C_{\bar{\mu}^{\sigma}}}(z-L^{*}_{\sigma})^{-1}dz.
	\end{eqnarray*}   
	Here, $E(\mu^{\sigma}):X_{per}(\tilde{\Omega}_H) \times X_{per}(\tilde{\Omega}_H) \rightarrow X_{per}(\tilde{\Omega}_H) \times X_{per}(\tilde{\Omega}_H)$, $E^{*}(\bar{\mu}^{\sigma}):X_{per}(\tilde{\Omega}_H) \times X_{per}(\tilde{\Omega}_H) \rightarrow X_{per}(\tilde{\Omega}_H) \times X_{per}(\tilde{\Omega}_H)$, and their ranges are
	\begin{equation*}
		R_{\mu^{\sigma}}:=\operatorname{Ker}(\mu^{\sigma}-L_{\sigma})^{r}, R^{*}_{\bar{\mu}^{\sigma}}:=\operatorname{Ker}(\bar{\mu}^{\sigma}-L^{*}_{\sigma})^{r}.
	\end{equation*}   
	Here, $r$ is the ascent of the operators $\mu^{\sigma}-L_{\sigma}$ and $\bar{\mu}^{\sigma}-L_{\sigma}^{*}$. In addition, a distance measure between two spaces is introduced
	\begin{equation*}
		d(V_{1},V_{2})=\sup_{v_{1}\in V_{1}}\inf_{v_{2}\in V_{2}}\frac{\Vert v_{1}-v_{2}\Vert _{H^{1}_{per}(\tilde{\Omega}_H)}}{\Vert v_{1}\Vert _{H^1_{per}(\tilde{\Omega}_H)}}.
	\end{equation*}    
	Therefore, it can be recorded as
	\begin{equation*}
		\gamma_{h}=d(R_{\mu_{\sigma}},V_{h}(\tilde{\Omega}_H)),~\gamma_{h}^{*}=d(R^{*}_{\mu_{\sigma}},V_{h}(\tilde{\Omega}_H)) .
	\end{equation*}    
	We assume that $R_{\mu^{\sigma}},R^{*}_{\bar{\mu}^{\sigma}} \subset H^2(\tilde{\Omega}_H)$. For first-order finite elements, the interpolation error estimates in \cite{Babuska, Brenner}, together with an argument analogous to the discussion of \cite[Equation (98)]{Engstrom1}, imply that there exist constants $C$ and $C^{*}$, depending only on the dimension and the domain $\tilde{\Omega}_{H}$, such that $\gamma_{h} \leq C h$ and $\gamma_{h}^{*} \leq C^{*} h$.
	Thus, let $N_{\sigma}=\dim R_{\mu^{\sigma}}$, and let $\mu^{\sigma}_{i}$, $i=1$, $\cdots$, $N_{\sigma}$ be the corresponding eigenvalues of $\mu^{\sigma}$ calculated by finite element discretization, then by \cite[Equation (95)]{Engstrom1}, we have
	\begin{equation}
		\max_{i=1,\cdots,N_{\sigma}}|\mu^{\sigma}-\mu^{\sigma}_{i}|\leq C (\gamma_h\gamma_{h}^*)^{\frac{1}{r}}\leq C h^{\frac{2}{r}}.\label{covrate} 
	\end{equation}     
	In our numerical experiments, we observe that the propagative values converge with a second‑order rate.

	\section{Numerical experiments} \label{sect7} 
	In this section, we consider two numerical experiments to verify the effectiveness of our method. We use the Matlab function $eigs$ to solve the eigensystem \eqref{disblockdiag} to compute the propagative wave numbers and the corresponding modes. First, in the following experiments, for the period $\Lambda=2\pi$, we choose the PML thickness to be $\delta=0.5$ and set $s(x_{2})$ accordingly, making the PML truncation error negligible compared to finite element discretization. In all experiments, the nonhomogeneous periodic medium of the Wigner-Seitz cell is in $[-\pi,\pi]\times[-0.5,0.5]$.  Since the periodic $\Lambda$ is $2\pi$, the propagative wave numbers are all in the interval $[-0.5, 0.5]$.  

    Since a guided mode remains a guided mode when multiplied by a constant, we compute its error as follows. We fix an exact solution $ u_{guide}^{e} $ and denote the numerically obtained eigenfunction as $ u_{guide}^{h} $. The $ L^2 $ error for the guided mode is then defined as  
$$
\min_{c \in \mathbb{C}} \bigl\Vert u_{guide}^{e} - c \, u_{guide}^{h} \bigr\Vert_{L^2(\Omega_H)}= \bigl\Vert u_{guide}^{e} - \frac{(u_{guide}^{e},u_{guide}^{h})_{L^2(\Omega_H)}}{(u_{guide}^{h},u_{guide}^{h})_{L^2(\Omega_H)}} \, u_{guide}^{h} \bigr\Vert_{L^2(\Omega_H)} .
$$
In essence, this treats the numerical solution as a hyperplane spanned by the computed eigenfunction and measures the $ L^2 $ error as the distance from the fixed exact solution to that hyperplane.

	%The regions and medium parameters in each experiment are as follows.
	%   \begin{enumerate}
		%    \item As shown in Figure \ref{exp1}, the medium region is a single rectangle with a constant density of $\gamma=9.8$.   
		%    \item As shown in Figure \ref{exp2}, the medium region is a single rectangle with a constant density of $\gamma=9$. 
		%    \item As shown in Figure \ref{exp3}, the medium region consists of two adjacent rectangles with different medium densities $\gamma_{1}=6$ and $\gamma_{2}=10$.       
		%   \end{enumerate}    
	
Note that the matrices $A^{h}_{\sigma,1}$ and $A^{h}_{\sigma,3}$ are complex symmetric, while $A^{h}_{\sigma,2}$ is skew‑symmetric. $\alpha$ is an eigenvalue of the discrete problem is equivalent to the condition that there exists a nonzero vector $\hat{x}$ satisfying
$$
\begin{bmatrix}
A_{\sigma,1}^h+tA_{\sigma,2}^h+t^2A_{\sigma,3}^h & 0 \\ 
0 & I^{h}
\end{bmatrix}\hat{x}
=(\alpha-t)
\begin{bmatrix}
-A_{\sigma,2}^h-2tA_{\sigma,3}^h & -A_{\sigma,3}^h \\ 
I^{h} & 0
\end{bmatrix}\hat{x}.
$$
A straightforward calculation shows that this is further equivalent to the condition
$$
\det\bigl(A_{\sigma,1}^h+\alpha A_{\sigma,2}^h+\alpha^2A_{\sigma,3}^h\bigr)=0.
$$
Thus,
$$
\det\bigl(A_{\sigma,1}^h-\alpha A_{\sigma,2}^h+\alpha^2A_{\sigma,3}^h\bigr)
= \det\Bigl((A_{\sigma,1}^h+\alpha A_{\sigma,2}^h+\alpha^2A_{\sigma,3}^h)^T\Bigr)=0.
$$
	This means that $\alpha$ and $-\alpha$ are the eigenvalues of the discrete problem simultaneously. Then, to find the propagative wave numbers, we only need to solve the problem \eqref{disblockdiag} to find all the eigenvalues with a very small imaginary part and whose real part lies in $[0,0.5]$.

In all numerical experiments of this paper, the shift parameter is taken as $ t = 0 $.

	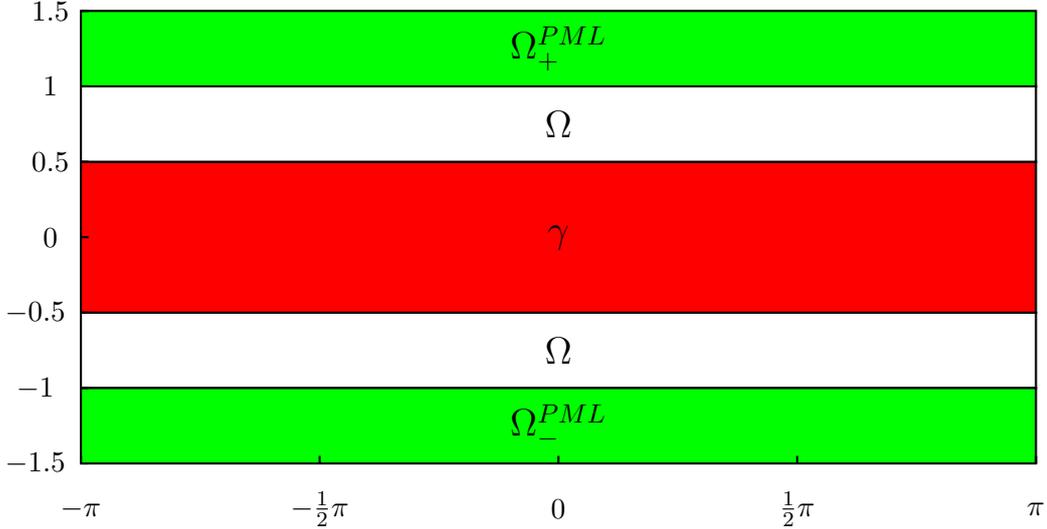
\begin{figure}[ht]
    \centering
		\begin{tikzpicture}[scale=1.5,thick]
			\filldraw[fill = white] (-pi,-1) rectangle (pi,-.5);
			\filldraw[fill = green] (-pi,-1.5) rectangle (pi,-1);
			\filldraw[fill = red] (-pi,-0.5) rectangle (pi,0.5);
			\filldraw[fill = white] (-pi,0.5) rectangle (pi,1);
			\filldraw[fill = green] (-pi,1) rectangle (pi,1.5);
			\draw[font = \Large] (0,1.25) node{$\Omega^{PML}_{+}$};
			\draw[font = \Large] (0,0.75) node{$\Omega$};
			\draw[font = \Large] (0,0) node{$\gamma$};
			\draw[font = \Large] (0,-1.25) node{$\Omega^{PML}_{-}$};
			\draw[font = \Large] (0,-0.75) node{$\Omega$};
			\draw (0,-1.5) -- (0,-1.45) node at(0,-1.8) {$0$};
			\draw (-pi,-1.5) -- (-pi+.05,-1.5) node at(-pi-0.3,-1.5) {$-1.5$};
			\draw (-pi,-1) -- (-pi+.05,-1) node at(-pi-0.3,-1) {$-1$};
			\draw (-pi,-.5) -- (-pi+.05,-.5) node at(-pi-0.3,-.5) {$-0.5$};
			\draw (-pi,0) -- (-pi+.05,0) node at(-pi-0.2,0) {$0$};
			\draw (-pi,.5) -- (-pi+.05,.5) node at(-pi-0.2,.5) {$0.5$};
			\draw (-pi,1) -- (-pi+.05,1) node at(-pi-0.2,1) {$1$};
			\draw (-pi,1.5) -- (-pi+.05,1.5) node at(-pi-0.2,1.5) {$1.5$};
			\foreach \x/\y in {-1/-,-0.5/-\frac{1}{2},0.5/\frac{1}{2},1/}
			\draw (\x*pi,-1.5) -- (\x*pi,-1.45) node at(\x*pi,-1.8) {$\y\pi$};
		\end{tikzpicture}
		\caption{Computational domain in Experiment 1.}\label{exp2}
	\end{figure}

    \subsection*{Experiment 1} The wave number $k=1.6$, and the refractive index $\gamma=9$ in the domain $(-\pi,\pi)\times(-0.5,0.5)$.   This example is inspired by the example given in \cite{Kirsch5}.  The propagative wave numbers can be explicitly calculated to be approximately $\alpha_1=0.4368$ and $\alpha_2=0.2911$, and the corresponding guided modes are of the form 
	\begin{equation*}
		u_{guide}^{(1)}(x)=e^{i\alpha_1x_1}\begin{cases}
			\begin{array}{cc}
				e^{-3ix_1}   \sin{\sqrt{\frac{\gamma k^2-(\alpha_1-3)^2}{2}}}e^{-\sqrt{(\alpha_1-3)^2-k^2}(x_2-\frac12)},   x_2>\frac{1}{2}, \\
				e^{-3ix_1} \sin{(\sqrt{\gamma k^2-(\alpha_1-3)^2}x_2)},  -\frac{1}{2}<x_2<\frac{1}{2}, \\
				-e^{-3ix_1}   \sin{\sqrt{\frac{\gamma k^2-(\alpha_1-3)^2}{2}}}e^{\sqrt{(\alpha_1-3)^2-k^2}(x_2+\frac12)},   x_2<-\frac{1}{2}.        
			\end{array}
		\end{cases}
	\end{equation*}
	
	\begin{equation*}
		u_{guide}^{(2)}(x)=e^{i\alpha_2x_1}\begin{cases}
			\begin{array}{cc}
				e^{4ix_1}   \cos{\sqrt{\frac{\gamma k^2-(\alpha_2+4)^2}{2}}}e^{-\sqrt{(\alpha_2+4)^2-k^2}(x_2-\frac12)},   x_2>\frac{1}{2}, \\
				e^{4ix_1} \cos{(\sqrt{\gamma k^2-(\alpha_2+4)^2}x_2)},  -\frac{1}{2}<x_2<\frac{1}{2}, \\
				e^{4ix_1}   \cos{\sqrt{\frac{\gamma k^2-(\alpha_2+4)^2}{2}}}e^{\sqrt{(\alpha_2+4)^2-k^2}(x_2+\frac12)},   x_2<-\frac{1}{2}.        
			\end{array}
		\end{cases}
	\end{equation*}
	The computing domain is shown in Figure \ref{exp2}, the PML parameter $$s(x_{2}) = 1+40*(1+i)\begin{cases}
		\begin{array}{cc}
			(\frac{x_{2}-1}{0.5})^3,  & x_2>1, \\
			(\frac{x_{2}+1}{0.5})^3,  & x_2<-1. 
		\end{array}.
	\end{cases}$$
	The convergence rates of these two pairs of propagative wave numbers are shown in Table \ref{tabula2} and Table \ref{tabula3}. The order shows an approximate second-order convergence in solving the open waveguide problem, which is in agreement with our theory.
	\begin{table}[ht]
    \centering
        		\caption{Numerical results and convergence orders for the propagative wave number $\alpha_1=0.4368$ and the corresponding guided mode $u_{guide}^{(1)}$ in Experiment 1.}
		\begin{tabular}{|c|c|c|c|c|}
			\hline
			$hmax$        &          $\alpha_1$ &    Order     &  $L^2$ error of $u_{guide}^{(1)}$&    Order  \\
			\hline
			$2^{-3}$         &0.5060 + 6E-03i & --- & 3.48E-01 & --- \\ 
			\hline
			$2^{-4}$         &0.4540 + 2E-03i   & 2.01 & 7.66E-02& 2.18 \\ 
			\hline
			$2^{-5}$         &0.4411 + 5E-04i  & 2.00 & 1.91E-02 &2.01 \\
			\hline      
			$2^{-6}$         &0.4379 + 1E-04i & 1.99 & 4.79E-03&1.99 \\ 
			\hline
			$2^{-7}$         &0.4371 + 3E-05i  &  1.98 & 1.20E-03 &1.99 \\
			\hline
		\end{tabular}
		\label{tabula2}
	\end{table}
	
	\begin{table}[ht]
    \centering
		\caption{Numerical results and convergence orders for the propagative wave number $\alpha_2=0.2911$ and the corresponding guided mode $u_{guide}^{(2)}$ in Experiment 1.}
		\begin{tabular}{|c|c|c|c|c|}
			\hline
			$hmax$        &          $\alpha_2$ &    Order     &   $L^2$ error of $u_{guide}^{(2)}$ &    Order  \\
			\hline
			$2^{-3}$         &0.2544 - 2E-04i & --- & 1.46E-01 & --- \\ 
			\hline
			$2^{-4}$         &0.2820 - 5E-05i	   & 1.99 & 3.64E-02& 2.01 \\ 
			\hline
			$2^{-5}$         &0.2888 - 1E-05i   & 2.00 & 9.07E-03&2.00 \\
			\hline      
			$2^{-6}$         &0.2905 - 4E-06i	 & 1.99 & 2.27E-03&2.00 \\ 
			\hline
			$2^{-7}$         &0.2909 - 9E-07i	  &  1.98 & 5.67E-04 &2.00 \\
			\hline
		\end{tabular}
		\label{tabula3}
	\end{table}

	\subsection*{Experiment 2} The wave number $k=\pi$. The refractive index in the periodic cell is piecewise constant, as shown in Figure \ref{exp3}. We set the refractive index $\gamma_1=6, \gamma_2=10$ in the cell $(-\pi,\pi)\times (-0.5,0.5)$. In this situation,  the explicit representations of the guided modes cannot be derived, and the propagative wave numbers and modes must therefore be computed numerically. 
	The computing domain is shown in Figure \ref{exp3}. We set the PML parameter 
	$$s(x_{2}) = 1+40*(1+i)\begin{cases}
		\begin{array}{cc}
			(\frac{x_{2}-1}{0.5})^3,  & x_2>1, \\
			(\frac{x_{2}+1}{0.5})^3,  & x_2<-1. 
		\end{array}
	\end{cases}$$ 
 	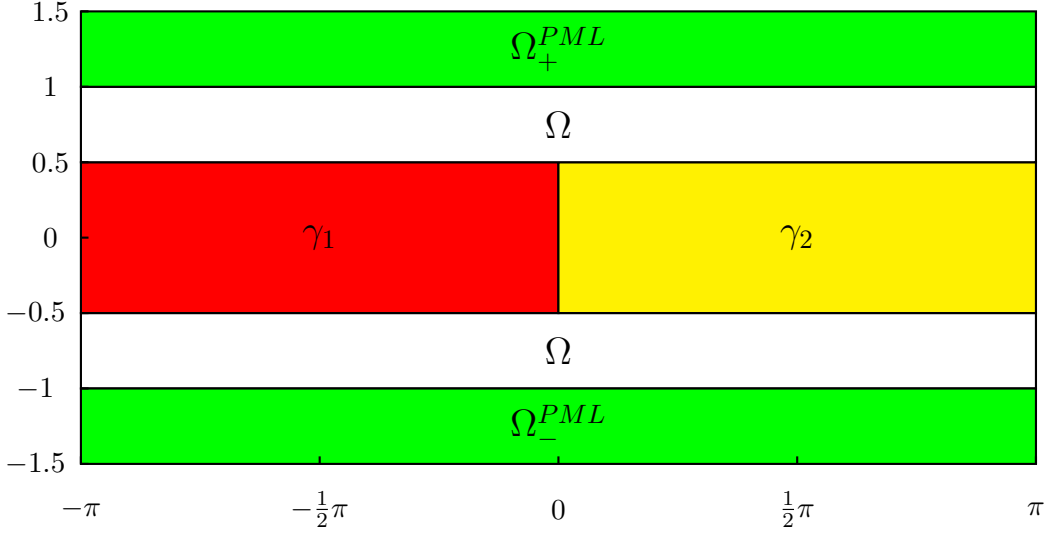
\begin{figure}[ht]
    \centering
		\begin{tikzpicture}[scale=1.5,thick]
			\filldraw[fill = white] (-pi,-1) rectangle (pi,-.5);
			\filldraw[fill = green] (-pi,-1.5) rectangle (pi,-1);
			\filldraw[fill = red] (-pi,-0.5) rectangle (0,0.5);
			\filldraw[fill = yellow] (0,-0.5) rectangle (pi,0.5);
			\filldraw[fill = white] (-pi,0.5) rectangle (pi,1);
			\filldraw[fill = green] (-pi,1) rectangle (pi,1.5);
			\draw[font = \Large] (0,1.25) node{$\Omega^{PML}_{+}$};
			\draw[font = \Large] (0,0.75) node{$\Omega$};
			\draw[font = \Large] (-pi/2,0) node{$\gamma_{1}$};
			\draw[font = \Large] (pi/2,0) node{$\gamma_{2}$};
			\draw[font = \Large] (0,-1.25) node{$\Omega^{PML}_{-}$};
			\draw[font = \Large] (0,-0.75) node{$\Omega$};
			\draw (0,-1.5) -- (0,-1.45) node at(0,-1.8) {$0$};
			\draw (-pi,-1.5) -- (-pi+.05,-1.5) node at(-pi-0.3,-1.5) {$-1.5$};
			\draw (-pi,-1) -- (-pi+.05,-1) node at(-pi-0.3,-1) {$-1$};
			\draw (-pi,-.5) -- (-pi+.05,-.5) node at(-pi-0.3,-.5) {$-0.5$};
			\draw (-pi,0) -- (-pi+.05,0) node at(-pi-0.2,0) {$0$};
			\draw (-pi,.5) -- (-pi+.05,.5) node at(-pi-0.2,.5) {$0.5$};
			\draw (-pi,1) -- (-pi+.05,1) node at(-pi-0.2,1) {$1$};
			\draw (-pi,1.5) -- (-pi+.05,1.5) node at(-pi-0.2,1.5) {$1.5$};
			\foreach \x/\y in {-1/-,-0.5/-\frac{1}{2},0.5/\frac{1}{2},1/}
			\draw (\x*pi,-1.5) -- (\x*pi,-1.45) node at(\x*pi,-1.8) {$\y\pi$};
		\end{tikzpicture}
		\caption{Computational domain in Experiment 2.}\label{exp3}
	\end{figure}
    
In this case, we verify the correctness of the guided modes by using numerical results on a very fine mesh, where the numerical modes correspond to the actual guided modes. The guided-mode screening method used in this experiment is the same as in Experiments 1 and 2. Specifically, we select the numerical guided modes whose propagative wave numbers have a much smaller imaginary part than the real part. When the meshsize $hmax$ is $2^{-8}$, all the eigenvalues with arguments lying in $[ -\arctan{0.1},\arctan{0.1}]$ and real parts lying in $[0, 0.5]$ are 0.2728 + 3E-03i and 0.4637 + 1E-04i.  We denote $\alpha_3=0.4637$ and $\alpha_4=0.2728$ as the two propagative wave numbers, and by $u_{guide}^{(3)}, u_{guide}^{(4)}$ the corresponding guided modes.
	% \begin{table}[ht]
		%  \centering
		%  \begin{tabular}{|l|l|}
			%  \hline
			%        0.2728 + 3E-0.3i &  0.4637 + 1E-04i\\
			%        \hline
			%  \end{tabular}
		%  \caption{All the computed eigenvalues with argument lying in $[-\arctan{0.1},\arctan{0.1}]$ and real part lying in $[0, 0.5]$ for $h_{max}=2^{-8}$ in Experiment 3. }
		%  \label{table_ex3}
		%\end{table}
		
		The convergence rates of these two pairs of propagative wave numbers are shown in Table \ref{tabula5} and Table \ref{tabula6}. Here we remark that, since exact propagative wave numbers and modes are not available, we just use the numerical result with meshsize $hmax=2^{-8}$ as the exact ones to compute the convergence rates. The order shows an approximate second-order convergence, which is in agreement with our theory. 
		
		As shown in Figures \ref{alpha23}, \ref{alpha45}, the guided modes calculated by our numerical method show rapid decay outside the medium layer.  These results further confirm the theory result in Corollary \ref{corollary2} that the guided modes decay exponentially outside the periodic structure. 
		
		\begin{table}[ht]
        \centering
		\caption{Numerical errors and convergence orders for the guided mode $\alpha_3=0.4637$ and the corresponding guided mode $u_{guide}^{(3)}$ in Experiment 2.}
			\begin{tabular}{|c|c|c|c|c|}
				\hline
				$hmax$        &          $\alpha_3$ &    Order     &  $L^2$ error of $u_{guide}^{(3)}$ &    Order  \\ \hline
				$2^{-3}$         &0.3241 - 3E-02i & --- & 2.70E-03 & --- \\ 
				\hline
				$2^{-4}$         &0.3962 + 4E-05i	   & 1.05 & 7.15E-04 & 1.92 \\ 
				\hline
				$2^{-5}$         &0.4446 + 4E-05i   & 1.82 & 3.20E-04 &1.16 \\
				\hline      
				$2^{-6}$         &0.4587 + 1E-04i	 & 1.94 & 9.98E-05 &1.68 \\ 
				\hline
				$2^{-7}$         &0.4627 + 1E-04i	  &  2.29 & 2.16E-05 &2.21 \\
				\hline
				$2^{-8}$         &0.4637 + 1E-04i	  &  --- & --- & --- \\
				\hline
			\end{tabular}
			\label{tabula5}
		\end{table}
		
		\begin{table}[ht]
        \centering
					\caption{Numerical errors and convergence orders for the guided mode $\alpha_4=0.2728$ and the corresponding guided mode $u_{guide}^{(4)}$ in Experiment 2.}
			\begin{tabular}{|c|c|c|c|c|}
				\hline
				$hmax$        &          $\alpha_4$ &    Order     &   $L^2$ error of $u_{guide}^{(4)}$ &    Order  \\
				\hline
				$2^{-3}$         &0.2163 + 2E-03i & --- & 2.89E-03 & --- \\ 
				\hline
				$2^{-4}$         &0.2030 + 3E-03i	   & -0.31 & 3.79E-04& 2.93 \\ 
				\hline
				$2^{-5}$         &0.2555 + 3E-03i   & 2.02 & 9.45E-05&2.01 \\
				\hline      
				$2^{-6}$         &0.2686 + 3E-03i	 & 2.07 & 2.20E-05&2.10 \\ 
				\hline
				$2^{-7}$         &0.2720 + 3E-03i	  &  2.36 & 4.30E-06 &2.36 \\
				\hline
				$2^{-8}$         &0.2728 + 3E-03i & --- & --- & --- \\ 
				\hline
			\end{tabular}
			\label{tabula6}
		\end{table}
		
		\begin{figure}[ht]
			\centering{
				%\subfloat[\label{alpha2}]
				{\includegraphics[width=0.45\textwidth]{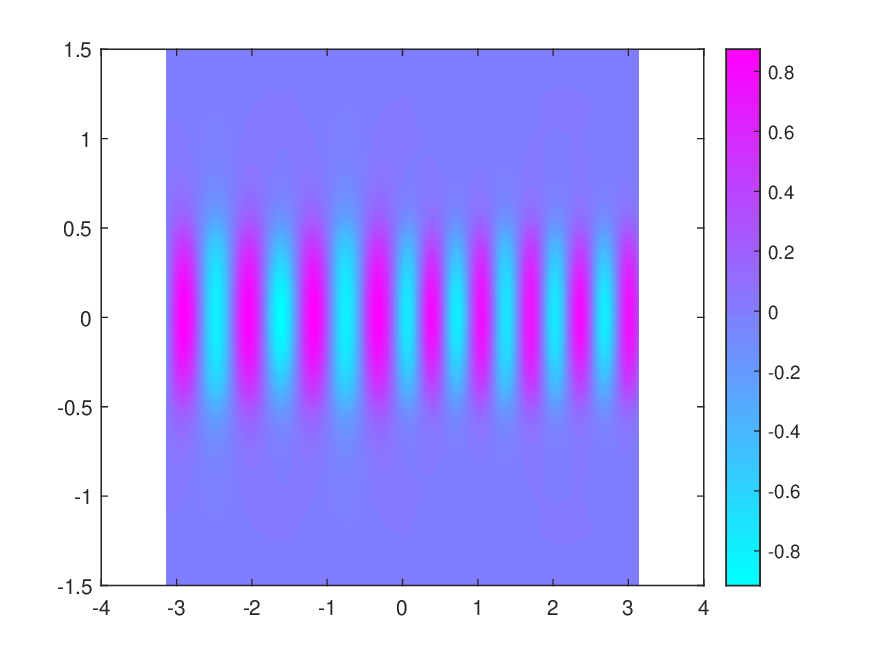}}
				%\subfloat[\label{alpha3}]
				{\includegraphics[width=0.45\textwidth]{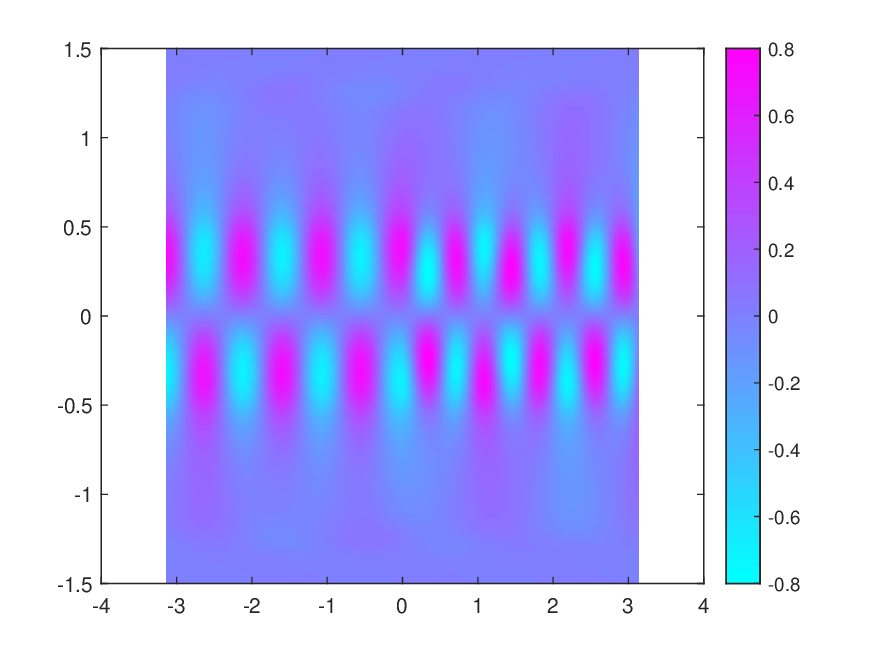}}
			}
			\caption{Numerical guided modes calculated in Experiment 2. Left:  $\alpha_3=0.4637$; Right: $\alpha_4=0.2728$.}
			\label{alpha23}
		\end{figure}   
		\begin{figure}[ht]
			\centering{
				%\subfloat[\label{alpha4}] 
				{\includegraphics[width=0.45\textwidth]{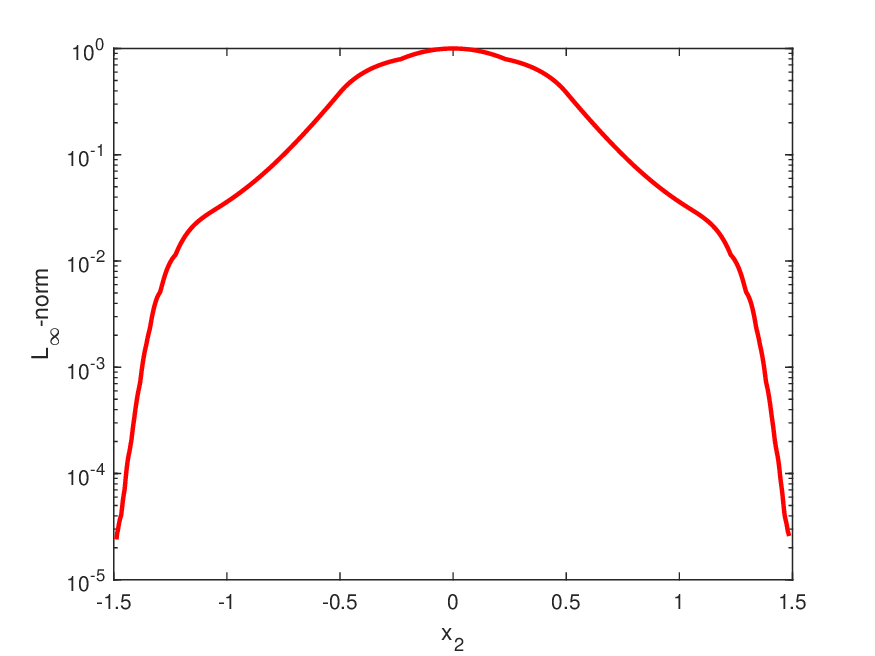}}
				%\subfloat[\label{alpha5}]
				{\includegraphics[width=0.45\textwidth]{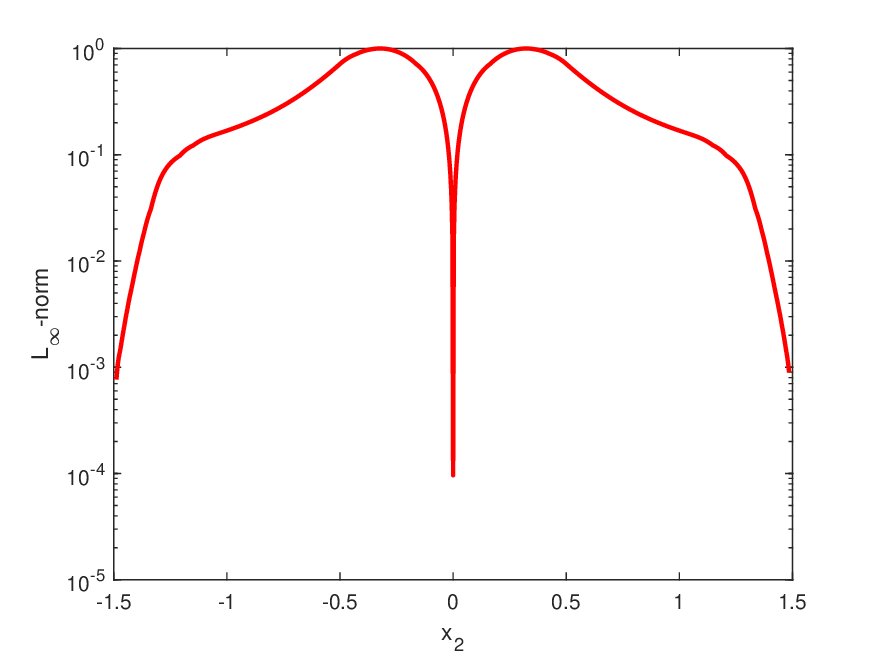}}
			}
			\caption{$\Vert u^{(3)}_{prop}(\cdot ,x_2) \Vert_{L^{\infty}}$(Left) and $\Vert u^{(4)}_{prop}(\cdot ,x_2) \Vert_{L^{\infty}}$(Right) calculated in Experiment 2.}
			\label{alpha45}
		\end{figure}  

\end{document}